\documentclass[oneside,a5paper,11pt]{amsart}
\usepackage{mathtools,tensor,physics,tikz-cd}
\newcommand\bydef{\coloneqq}
\usepackage[a4paper,scale=.75]{geometry}
\usepackage{newtx,microtype,enumitem,graphicx}
\usepackage[mathscr]{eucal}
\usepackage[hidelinks]{hyperref}
\usepackage[alphabetic]{amsrefs}
\usepackage{tikz}
\usetikzlibrary{arrows.meta, positioning}
\numberwithin{equation}{section}
\newtheorem{THEO}{Theorem}
\newenvironment{THEO'}[1]{\renewcommand{\theTHEO}{\ref*{#1}'}\begin{THEO}}{\end{THEO}\addtocounter{THEO}{-1}}
\newtheorem{theo}{Theorem}[section]

\newtheorem{lemm}[theo]{Lemma}
\newtheorem{coro}[theo]{Corollary}
\theoremstyle{definition}

\theoremstyle{remark}
\newtheorem{exam}[theo]{Example}
\newtheorem{rema}[theo]{Remark}
\newcommand{\R}{\mathbb{R}}
\DeclareMathOperator{\GL}{GL}
\DeclareMathOperator{\SL}{SL}
\DeclareMathOperator{\GO}{O}
\DeclareMathOperator{\SO}{SO}
\DeclareMathOperator{\Mat}{Mat}
\DeclareMathOperator{\Hom}{Hom}
\DeclareMathOperator{\End}{End}
\DeclareMathOperator{\Sym}{Sym}
\DeclareMathOperator{\Vect}{Vect}
\newcommand{\pol}{\mathcal{P}}
\newcommand{\rest}{\mathcal{R}}
\DeclareMathOperator{\id}{id}
\DeclareMathOperator{\diag}{diag}

\DeclarePairedDelimiter{\Set}{\lbrace}{\rbrace}
\DeclareMathOperator{\rey}{\mathscr{R}}
\def\O{\mathscr{O}}
\author{L.~Darondeau}
\author{M.~Florence}
\author{B.~Kolev}
\title[Weyl's Polarization in Classical Invariant Theory]{Weyl's Polarization in Classical Invariant Theory:\\A Primer, with Worked Examples}
\thanks{We are grateful to the GDR CNRS n° 2043 ``Géométrie Différentielle et Mécanique'' for fostering a supportive environment and for hosting the early presentation of this work.}
\begin{document}
\begin{abstract}
  Hermann Weyl’s \textit{The Classical Groups} is a landmark work connecting classical invariant theory with modern representation theory. It shows how polynomial invariants of the general linear, orthogonal, and symplectic groups can be systematically understood through linear representations and tensor methods.
  The current note is primarily based on a personal reading of the book of Weyl and of the more accessible books \textit{Classical Invariant Theory} by Kraft and Procesi and by Olver.
  It is neither exhaustive, nor original, nor state of the art. 
  We focus on a few selected aspects, aiming for an elementary and concrete approach.
  We work over the field of reals \(\R\) with the classical groups \(\GL(n)\), \(\SL(n)\), \(\GO(n)\), and \(\SO(n)\).
  Most of our efforts have been devoted to carefully worked examples, introducing just enough of the general theory to handle them effectively.
\end{abstract}
\maketitle

\numberwithin{section}{part}
\setcounter{tocdepth}{2}
\tableofcontents
\renewcommand{\thepart}{\Roman{part}}
\part{Introduction: from algebraic to linear problems.}
\section{Invariant \texorpdfstring{\(n\)}{n}-ary forms.}\label{secinv}
A \textsl{linear representation} of a group \(G\) on a vector space \(V\) is a group homomorphism \(\rho\colon G\to \GL(V)\). A \textsl{subrepresentation} of \(V\) is a vector subspace \(W\) that is stable under the action of \(G\). It is a representation in itself, for the induced action \(G\to\GL(W)\). For a linear representation of a group \(G\) on a vector space \(V\), and a subset \(S\subseteq V\), there is a smaller subrepresentation containing \(S\). It is obtain in two steps: first taking the smallest subspace of \(V\) containing \(S\) and stable under the action of \(G\), which is the reunion \(\rho(G)(S)\) of all orbits, and then taking the vector subspace that it generates, which is still stable because \(\rho(g)\) is a linear map.
To sum up: the subrepresentation generated by \(S\) is the subspace
\[
  \langle G\ast S\rangle
  \bydef
  \Vect\Set{\rho(g)(\vb*{v}),g\in G,\vb*{v}\in S}.
\]

Let \(G\) be a \textsl{linear group}, that is a closed matrix group \(G\subseteq\GL(r)\).
A representation \(\rho\colon G\to\GL(V)\) of \(G\) is said \textsl{polynomial} if \(\rho(g)(\vb*{v})\) depends polynomially on the coefficients of \(g\). 
It is said \textsl{rational} if this dependence is rational.
More precisely, these properties can be checked in coordinates using the following definitions.
A \textsl{polynomial} function \(G\to\R\) is the restriction to \(G\subseteq\operatorname{M}_r\) of a polynomial function \(f\in\R[\operatorname{M}_r]\) in the coefficients of \((r\times r)\)-matrices \(g\in\operatorname{M}_r\).
A \textsl{regular} function \(G\to\R\) is the restriction to \(G\) of a rational function \(f\in\R(\operatorname{M}_r)\) that is defined everywhere on \(G\).
We denote by \(\O(G)\) the algebra of regular functions \(G\to\R\). We use this notation rather than the standard mathematical notation \(\R[G]\) in order to avoid a possible confusion with the algebra of polynomial functions on \(G\).
Note that a polynomial that does not vanish on the full matrix group \(\GL(r)\) is a power of \(\det(g)\);
this is a standard consequence of Hilbert's Nullstellensatz, a fundamental result in algebraic geometry, and it will be admitted.
Therefore, for a subgroup \(G\subseteq\GL(r)\), the denominators of regular functions \(f\in\O(G)\) are always powers of \(\det\). 
The foregoing discussion can be summarized by the following equation:
\[
  \O(G)
  \subseteq
  \O(\GL(r))
  =
  \R[\operatorname{M}_r][\operatorname{det}^{-1}].
\]

Let \(V\) be a \(n\)-dimensional rational representation of the linear group \(G\). Then the direct sum 
\[
  V^m
  =
  V\oplus\stackrel{m}\dotsb\oplus V
\]
of \(m\) copies of \(V\) is a representation as well, for the \textsl{diagonal action}
\[
  g\ast(\vb*{v}_1,\dotsc,\vb*{v}_m)
  = 
  (\rho(g)(\vb*{v}_1),\dotsc,\rho(g)(\vb*{v}_m)).
\]
Given a basis \(\mathcal{B}=(\vb*{e}_1,\dotsc,\vb*{e}_n)\) of \(V\), an element \((\vb*{v}_1,\dotsc,\vb*{v}_m)\) of \(V^m\) is represented by the \((n\times m)\)-matrix of coordinates
\[
  \vb{M}
  =
  \Mat(v_\alpha^i)
  =
  \left[\vb*{v}_1\vert\dotso\vert \vb*{v}_m\right]_{\mathcal{B}},
\]
whose columns are the coordinates of the vectors \(\vb*{v}_\alpha\).
For \(g\in G\), let \(A\) be the \((n\times n)\)-matrix of the linear isomorphism \(\rho(g)\colon V\to V\) in the basis \(\mathcal{B}\).
Then the matrix form of the diagonal action is the matrix product
\[
  [g*\vb{M}]_{\mathcal{B}}
  =
  \left[A\vb*{v}_1\vert\dotso\vert A \vb*{v}_m\right]_{\mathcal{B}}
  =
  A\vb{M}.
\]

This action induces a rational representation of \(G\) on the vector space \(\R[V^m]\) of multivariate polynomials \(f=f(\vb{M})=f(v_\alpha^i)\), given by
\[
  (g\ast f)(\vb{M})
  =
  f(g^{-1}\ast \vb{M}).
\]
This representation preserves the decomposition of \(f\in\R[V^m]\) into the sum of homogeneous polynomials \(f=f_{(0)}+f_{(1)}+\dotsb+f_{(d)}\), where the degree of a monomial is its total degree in the components \(v_\alpha^i\).
An \textsl{invariant \(m\)-ary form} is a homogeneous multivariate polynomial \(f\in\R[V^m]\) such that \(g\ast f = f\) for every \(g\in G\).
In other words, the homogeneous polynomial \(f\) must be constant on the orbits of \(G\) in \(V^m\).

The main goal of Classical Invariant Theory is to classify and describe polynomial invariants,
understand their algebraic relations, and determine generating sets for the invariant algebras. 
Explicit computation of these invariants soon becomes tedious.
The symbolic approach represents forms and invariants using formal symbols, allowing one to manipulate compact ``human readable'' algebraic expressions without immediately expanding them into intractable expressions.
It translates the action of a group on polynomials into operations on symbols, making the computation of invariants more systematic and combinatorial. This method was a central tool before the modern, tensor-based representation-theoretic perspective.
Weyl has connected the classical symbolic approach with modern representation theory by showing that the formal symbols of invariant theory correspond to tensors acted on by groups. Invariants then appear naturally as elements fixed under these group actions, allowing symbolic manipulations to be replaced by systematic, linear-algebraic methods.

\section{Polarization and restitution.}
The main tool to turn an algebraic problem into a linear-algebraic problem is the polarization process.
The most familiar view of polarization is as the process of linearizing a homogeneous polynomial into a symmetric multilinear form.

\begin{theo}
  Every homogeneous polynomial \(f(\vb*{v})\in\R[V]\) of degree \(d\) can be uniquely associated to a symmetric multilinear form \(\pol_d(f)(\vb*{v}_1,\dotsc,\vb*{v}_d)\in\Sym^d(V^*)\subseteq\R[V^d]\), that is called its \textsl{complete polarization}.
  Moreover, the original polynomial \(f\) is recovered from \(\pol_d(f)\) by evaluating it on the diagonal: 
  \[
    f(\vb*{v})=\pol_d(f)(\vb*{v},\dotsc,\vb*{v}).
  \]
\end{theo}
\begin{proof}
  For a degree \(d\) monomial \(f(\vb*{v})=v^{i_1}\dotsm v^{i_d}\), we define \(\pol_d(f)\) by averaging on the symmetric group:
  \[
    \pol_d(f)(\vb*{v}_1,\dotsc,\vb*{v}_d)
    =
    \frac{1}{d!}
    \sum_{\alpha\in S_d}
    v_{\alpha_1}^{i_{1}}\dotsm v_{\alpha_d}^{i_{d}}
    \in\Sym^d(V^*).
  \]
  Then we extend the definition by linearity.
\end{proof}
\begin{rema}
  If \(F\colon V^{d}\to\R\) is multilinear, its evaluation on the diagonal is a homogeneous polynomial \(\rest_d(F)\) of degree \(d\).
  The process \(F\to\rest_d(F)\) is called the \textsl{restitution}.
  Therefore, for any \(F\), after polarization, \(\pol_d(\rest_d(F))\) is again a multilinear form on \(V^{d}\), but which is symmetric.  It agrees with \(F\) if and only if \(F\) is already symmetric. Hence the linear map \((\pol_d\circ\rest_d)\) is a projection on the subspace of symmetric multilinear forms.
\end{rema}
\begin{exam}
  Let \(\ell\) be a linear form on \(V\).
  One has
  \[
    \pol_d(\ell^d)(\vb*{v}_1,\dotsc,\vb*{v}_d)
    =
    \ell(\vb*{v}_1)\dotsm\ell(\vb*{v}_d).
  \]
  Indeed \(\rest_d(\ell(\vb*{v}_1)\dotsm\ell(\vb*{v}_d))=\ell(\vb*{v})^d\), and the multilinear form
  \(\ell(\vb*{v}_1)\dotsm\ell(\vb*{v}_d)\) is symmetric.
\end{exam}
\begin{rema}
  Since we use the diagonal action of the group \(G\) on the vector space \(V^d\), the polarization process \(\pol_d\), the restitution process \(\rest_d\), and the projector \((\pol_d\circ\rest_d)\) all commute with the action of \(G\).
\end{rema}
\begin{exam}
  Let \(V=\R^2\), and let \(f(\vb*{v})=av^1v^1+2bv^1v^2+cv^2v^2\) be a quadratic form.
  Its polarization is the symmetric bilinear form
  \[
    \pol_2(f)(\vb*{v}_1,\vb*{v}_2)
    =
    av_1^1v_2^1+b(v_1^1v_2^2+v_2^1v_1^2)+cv_1^2v_2^2.
  \]
\end{exam}
\begin{exam}
  Let \(V=\R^3\), and let \(f(\vb*{v})=3v^1v^1v^3\), a cubic form.
  Its polarization is the symmetric trilinear form
  \[
    \pol_3(f)(\vb*{v}_1,\vb*{v}_2,\vb*{v}_3)
    =
    v_1^1v_2^1v_3^3
    +
    v_1^1v_2^3v_3^1
    +
    v_1^3v_2^1v_3^1.
  \]
\end{exam}

We will need a more general version with several representations.
\begin{theo}
  Let \(V_1,\dots,V_k\) be finite-dimensional representations of a group \(G\).  
  Every multi-homogeneous polynomial \(f\) of multi-degree \((d_1,\dots,d_k)\) in vectors \(\vb*{v}_1,\dots,\vb*{v}_k\) can be uniquely associated with a multi-symmetric multilinear form
  \[
    \pol_{d_1,\dotsc,d_k}(f)\colon V_1^{d_1}\times\dotsb\times V_k^{d_k} \to \R.
  \]
  Moreover, the original polynomial \(f\) is recovered from \(\pol_{d_1,\dotsc,d_k}(f)\) by evaluating it on the multi-diagonal. 
\end{theo}
\begin{proof}
  A multi-homogeneous polynomial \(f\) of multi-degree \((d_1,\dotsc,d_k)\) can be written as a finite sum
  \[
    f(\vb*{v}_1,\dotsc,\vb*{v}_k)
    =
    \sum_i
    h_i^1(\vb*{v}_1)
    \dotsm
    h_i^k(\vb*{v}_k),
  \]
  where the functions \(h_i^j\) are homogeneous of degree \(d_j\).
  In tensor notations on the polynomial algebra 
  \[
    \R[V_1\times\dotsb\times V_k]=\R[V_1]\otimes\dotsb\otimes\R[V_k],
  \]
  one has
  \[
    f
    =
    \sum_i
    h_i^1\otimes\dotsb\otimes h_i^k.
  \]
  We use the complete polarization operator \(\pol_{d_j}\) on each \(\R[V_j]\):
  \[
    \pol_{d_1,\dotsc,d_k}(f)
    \bydef
    \sum_i
    \pol_{d_1}(h_i^1)
    \otimes
    \dotsb
    \otimes
    \pol_{d_k}(h_i^k).
  \]
  The restitution is then given by \(\rest_{d_1}\otimes\dotsb\otimes\rest_{d_k}\), that is by the evaluation on the multi-diagonal.
\end{proof}
\begin{exam}
  Let \(V_1=\R^2\) and \(V_2=\R^3\). Let \(f(\vb*{u},\vb*{v})=3u^1u^1v^1v^1v^3\), a quintic monomial of the vectors \(\vb*{u}\in V_1\) and \(\vb*{v}\in V_2\).
  Then \(\pol_{2,3}(f)\) is the multilinear form in \(5\) vector variables
  \[
    \pol_{2,3}(f)(\vb*{u}_1,\vb*{u}_2,\vb*{v}_1,\vb*{v}_2,\vb*{v}_3)
    =
    u_1^1u_2^1
    (
    v_1^1v_2^1v_3^3
    +
    v_1^1v_2^3v_3^1
    +
    v_1^3v_2^1v_3^1
    ).
  \]
\end{exam}

\begin{rema}
  If \(F\) is a multilinear form on  \(V_1^{d_1} \times \dotsb \times V_k^{d_k}\), its evaluation on the multi-diagonal is a multi-homogeneous polynomial \(\rest_{d_1,\dotsc,d_k}(F)\) in \(\vb*{v}_1,\dotsc,\vb*{v}_k\), of multi-degree \((d_1,\dotsc,d_k)\).
  The process \(F\to\rest_{d_1,\dotsc,d_k}(F)\) is called the \textsl{restitution}.
  As in the case \(k=1\),
  the linear map \((\pol_{d_1,\dotsc,d_k}\circ\rest_{d_1,\dotsc,d_k})\) is a projection on the subspace of multi-symmetric multilinear forms,
  which commutes with the action of \(G\).
\end{rema}

Since the polarization and the restitution processes commute with the action of \(G\),
invariant multi-homogeneous polynomials are sent on invariant multilinear forms.
This leads to the following final formulation.
\begin{THEO}\label{theo:linearization}
  Let \(V_1,\dotsc,V_k\) be finite-dimensional representations of a group \(G\).
  One can identify multi-homogeneous invariant forms with multi-symmetric invariant multilinear maps.
\end{THEO}

In what follows, we will mainly consider several copies \(V_1,\dotsc,V_k\) of a finite-dimensional representation \(V\), its dual representation \(V^*\) or the representation \(\End(V)\) with the action by conjugation.
Recall that there is a natural identification between the representations \(\End(V)\) and \(V\otimes V^*\).
Therefore, a linear form on the tensor space \(\End(V)=V\otimes V^*\) is a bilinear form on the product \(V^*\times V\).
More generally, every multilinear form \(\End(V)^d\to\mathbb{R}\) can be seen as a multilinear form \(V^d\oplus (V^*)^d\to\R\) in twice as many variables.
Let \(V_1,\dotsc,V_k\) be several copie of \(V\), \(V^*\) or \(\End(V)\).
By linearization, we can see every multihomogeneous polynomial form on \(V_1\oplus\dotsm\oplus V_k\) as 
certain multilinear maps on \(V,V^*,\End(V)\). By the identification \(\End(V)=V\otimes V^*\) above, we can actually see them as certain multilinear maps on \(V\) and \(V^*\) only (that is mixed tensors on \(V\)).
Since the polarization and the restitution processes commute with the \(G\)-action, this identification also identifies invariants.
Therefore \emph{the study of invariant mixed tensors provides a complete understanding of invariant forms on all spaces constructed from \(V\), \(V^*\), and \(\mathrm{End}(V)\)}.
\begin{exam}
  We take \(V_1=\End(V)\), and the quadratic form \(f(\vb{A})=\Trace(\vb{A}^2)\).
  Then, \(\pol_2(f)\) is the bilinear map
  \[
    \pol_2(f)(\vb{A},\vb{B})
    =
    \Trace(\vb{AB})
    =
    a_j^i b_i^j
    =
    \delta^j_k\delta^{\ell}_i
    a_j^i b_\ell^k.
  \]
  It corresponds to the fourth order mixed tensor with components
  \[
    \pol_2(f)(\vb*{e}_i,\vb*{e}^j,\vb*{e}_k,\vb*{e}^\ell)
    =
    \pol_2(f)(\vb*{e}_i\otimes \vb*{e}^j,\vb*{e}_k\otimes \vb*{e}^\ell)
    =
    \delta^j_{k}\delta^{\ell}_i.
  \]
\end{exam}

\section{Polarization operators.}
There is another more general description of the polarization process using differential operators of \(\R[V^d]\), as exposed notably in~\cites{GY03,O99}.
It is a convenient formulation of the classical symbolic approach.

We start with a preliminary fundamental result of independent interest, that will allow more direct proofs.
\begin{theo}\label{theo:pure_powers}
  The space \(\R[V]_d\) of homogeneous polynomials of degree \(d\) is spanned by the \(d\)th powers of linear forms.
\end{theo}
\begin{proof}
  It is sufficient to show this property for monomials \(v^I=(v^1)^{i_1}\dotsm(v^n)^{i_n}\).
  Consider the linear form \(\ell(t_1,\dotsc,t_n)=(t_1v^1+\dotsb+t_nv^n)\).
  One has
  \[
    \ell^d
    =
    \sum_{i_1+\dotsb+i_n=d}
    \frac{d!}{i_1!\dotsm i_n!}
    t^Iv^I.
  \]
  The monomials \(v^I\) can be retrieved using the finite differences operators
  \[
    \Delta_i(P)
    =
    P(t_1,\dotsc,t_i+1,\dotsc,t_n)-P(t_1,\dotsc,t_i,\dotsc,t_n).
  \]
  Indeed, for every \(j\), one has
  \[
    \Delta_j t_j^k=(t_j^{k-1}+\dotsb+t_j+1)
    \qqtext*{, whence}
    \Delta_j^kt_j^k=k!
    \qand
    \Delta_j^{k+1}t_j^k=0.
  \]
  Therefore
  \[
    \Delta_1^{i_1}\dotsm\Delta_n^{i_n}(\ell^d)(0)
    =
    d!
    v^I.
  \]
  Note that the polynomial is constant in \(t_1,\dotsc,t_n\) but evaluating at \(0\) can be useful in practice.
  We have written \(v^I\) as a linear combination of \(2^n\) pure powers.
\end{proof}
\begin{exam}\label{exam:pure_powers}
  One has the famous identity
  \[
    v^1v^2
    =
    \frac{ (v^1+v^2)^2 - (v^1)^2 - (v^2)^2 }{2}.
  \]
  One has also the less familiar identity
  \[
    v^1v^1v^2
    =
    \frac{(2v^1+v^2)^3-2(v^1+v^2)^3-6(v^1)^3+(v^2)^3}{6}.
  \]
  Actually, it can be derived by specialization from the more general (and more regular) formula
  \[
    v^1v^2v^3
    =
    \frac{(v^1+v^2+v^3)^3-(v^1+v^2)^3-(v^1+v^3)^3-(v^2+v^3)^3+(v^1)^3+(v^2)^3+(v^3)^3}{6}.
  \]
\end{exam}

We return to the main subject of this section.
For \(1\leq\alpha\leq d\), let \(\sigma_{\alpha}\) be the linear differential operator \(\R[V]\to\R[V\oplus V^d]\) defined by
\[
  \sigma_{\alpha}
  =
  \sum_{i=1}^{n}
  v_\alpha^i\pdv{v^i}.
\]
Note that the operators \(\sigma_\alpha,\sigma_\beta\) commute.
By Leibniz rule, for a linear form \(\ell\), one has 
\[
  \sigma_\alpha(\ell(\vb*{v})^d)
  =
  \ell(\vb*{v}_\alpha)d\ell(\vb*{v})^{d-1}.
\]
Therefore one recovers the complete polarization process as
\[
  \pol_d
  =
  \frac{1}{d!}\prod_{\alpha=1}^d\sigma_{\alpha}
\]
since these linear differential operators coincide on pure powers, which generate the vector space \(\R[V]_d\).

More generally, there is \textsl{partial polarization} processes \(\R[V]_d\to\R[V^d]\), that yields multi-homogeneous polynomials, defined for \(i_1,\dotsc,i_d\geq0\) such that \(i_1+\dotsb+i_d=d\) by
\[
  \pol_{(i_1,\dotsc,i_d)}
  =
  \frac{1}{d!}
  \prod_{\alpha=1}^d\sigma_{\alpha}^{i_\alpha}.
\]
Concretely, for pure powers, one has
\[
  \pol_{(i_1,\dotsc,i_d)}(\ell(\vb*{v})^d)
  =
  \ell(\vb*{v}_{1})^{i_1}
  \dotsm
  \ell(\vb*{v}_{d})^{i_d},
\]
and one extends this definition by linearity to all homogeneous polynomials in \(\R[V]_d\).

One can pass more conveniently from one polarization process to another using the linear differential operator \(\R[V^d]\to\R[V^d]\) defined for \(1\leq\alpha,\beta\leq d\) by
\[
  \sigma_{\alpha\beta}
  =
  \sum_{i=1}^{n}
  v_\alpha^i\pdv{v_\beta^i}.
\]
One can generate all polarization \(\pol_{(i_1,\dotsc,i_d)}\) processes from the complete polarization \(\pol_d=\pol_{(1,\dotsc,1)}\) by iterating the following rule, which is easily checked on pure powers.
If \(\alpha\neq\beta\), then one has:
\[
  \sigma_{\alpha\beta}\pol_{(i_1,\dotsc,i_d)}
  =
  i_\beta\pol_{(_1,\dotsc,j_d)},
\]
where \(j_\alpha=i_\alpha+1\), \(j_\beta=i_\beta-1\) and \(j_\gamma=i_\gamma\) for other indices \(\gamma\).

%One infers that the linear map
%\[
%  \R[\sigma_{\alpha\beta}]
%  \to
%  \langle\pol_{(I)},\abs{I}=d\rangle
%  \qcomma
%  Q(\sigma_{\alpha\beta})
%  \mapsto
%  Q(\sigma_{\alpha\beta})\cdot\pol_d
%\]
%is surjective.
\begin{exam}
  For \(d=3\), one has:
  \[
    \pol_{(3,0,0)}
    =
    \sigma_{12}\sigma_{13}\pol_{(1,1,1)}
    \qand
    \pol_{(2,0,1)}
    =
    \sigma_{12}\pol_{(1,1,1)}.
  \]
\end{exam}

\begin{rema}
  For \(\alpha=\beta\), the operation of \(\sigma_{\alpha\alpha}\) is just a rescaling (by the homogeneous degree in \(\vb*{v}_\alpha\)). 
  %We will from now on call these operators \(\sigma_{\alpha\alpha}\) \textsl{scaling operators} and reserve the terminology \textsl{polarization operator} for \(\sigma_{\alpha\beta}\) with \(\alpha\neq\beta\).
\end{rema}

\begin{rema}
  Let \(F\) be a multilinear form. Let \(\Delta_I\) be the set of vectors
  \[
    (\vb*{v}_1,\stackrel{i_1}\dotsc,\vb*{v}_{1},\dotsc,\vb*{v}_d,\stackrel{i_d}\dotsc,\vb*{v}_d).
  \]
  Evaluating \(F\) on \(\Delta_I\), we obtain a multi-homogeneous polynomial in \(\R[V^d]\) that we have denoted \(\rest_I(F)\).
  One has then
  \[
    \rest_I(\pol(f))
    =
    \pol_{(I)}(f).
  \]
  This is easily checked on pure powers.
  So the partial polarization can be obtained either from \(f\) by polarization or from the complete polarization \(F=\pol(f)\) by restriction to the multi-diagonal.
\end{rema}

\begin{exam}
  Let us follow up on our example cubic form
  \(f(\vb*{v})=3v^1v^1v^3\) on \(\R^3\).
  Its complete polarization is the symmetric trilinear form
  \[
    \pol_{(1,1,1)}(f)(\vb*{v}_1,\vb*{v}_2,\vb*{v}_3)
    =
    v_1^1v_2^1v_3^3
    +
    v_1^1v_2^3v_3^1
    +
    v_1^3v_2^1v_3^1.
  \]
  Then
  \[
    \pol_{(2,0,1)}(f)(\vb*{v}_1,\vb*{v}_2,\vb*{v}_3)
    =
    2v_1^1v_1^3v_3^1
    +
    v_1^1v_1^1v_3^3.
  \]
  and
  \[
    \pol_{(3,0,0)}(f)(\vb*{v}_1,\vb*{v}_2,\vb*{v}_3)
    =
    3v_1^1v_1^1v_1^3. 
  \]
\end{exam}

The relationships among the various processes are summarized in the following commutative diagram, for \(d=i_1+\dotsb+i_d\).
\[
  \begin{tikzcd}
&\R[V]_d
\ar[ldd,bend right=10,shorten >=10pt,swap,"\pol_{(i_1,\dotsc,i_d)}"]
\ar[rdd,bend left=10,"\pol_d"]
\ar[rrdd,bend left=10,"\pol_{(1,\dotsc,1)}"]
&
\\\\
\R[V]_{i_1}\otimes\dotsb\otimes\R[V]_{i_d}
%\ar[uur,hook,bend right=5]
\ar[rdd,swap,"\pol_{i_1,\dotsc,i_d}"]
&&
\Sym^d(V^*)
\ar[uul,bend left=5,"\rest_d"]
\ar[ll,"\rest_{i_1,\dotsc,i_d}"]
\ar[r,hook]
&
{\R[V]_1}^{\otimes d}.
\\\\
&
\Sym^{i_1}(V^*)\otimes\dotsb\otimes\Sym^{i_d}(V^*)
\ar[uur,hook]
  \end{tikzcd}
\]

In particular, one can read the interesting identity, for every integer partition \(d=i_1+\dotsb+i_d\):
\[
  \pol_d
  =
  \pol_{i_1,\dotsc,i_d}\circ\pol_{(i_1,\dotsc,i_d)}.
\]

%\begin{rema}
%From this diagram, one infers that for any multilinear form \(F\), the partial restitution \(\rest_{I}(F)\) is multi-homogeneous in \(d\) variables. We can apply full polarization to each factor as in the end of the previous section, in order to obtain a multi-symmetric multilinear form \((\pol_I\circ\rest_I)(F)\). It is equal to \(F\) if and only if \(F\) was already multi-symmetric.
%\end{rema}

\section{Polarization as an infinitesimal action.}
We will now present a second way to describe the polarization operators, as the generators of the infinitesimal action of a general linear group, as exposed for instance in~\cite{KP96}.
We have already seen that, given coordinates on \(V\), elements of \(V^m\) are identified with \((n\times m)\)-matrices \(\vb{M}\).
The diagonal action of the group \(G=\GL(n)\) on \(\vb{M}\) is by the matrix product \(g\ast \vb{M} = g\vb{M}\).
There is another natural linear action of the general linear group \(\Gamma=\GL(m)\) given by the matrix product \(\gamma\ast \vb{M}\bydef \vb{M}\gamma^{-1}\). It is an operation on the columns of \(\vb{M}\).
This action induces a rational representation of \(\Gamma\) on the vector space \(\R[V^m]\) of multivariate polynomials in vectors, which preserves the total degree, given by
\[
  (\gamma\ast f)(\vb{M})\bydef f(\vb{M}\gamma).
\]
\begin{exam}
  For \(m=2\):
  \[
    \qq*{if} \gamma=\begin{pmatrix}1&2\\3&1\end{pmatrix}
    \qcomma
    \gamma\ast f(\vb*{v}_1,\vb*{v}_2)
    =
    f(\vb*{v}_1+3\vb*{v}_2,2\vb*{v}_1+\vb*{v}_2).
  \]
\end{exam}

Let \(E\) be a differentiable \(\Gamma\)-representation.
For a \(1\)-parameter subgroup \(\gamma(s)\) of \(\Gamma\) parametrized by \(s\in(\R,+)\), we consider the following linear differential operator on \(E\):
\[
  L_\gamma
  \colon
  x
  \mapsto
  \partial_s\left(\gamma(s)\ast x\right)\vert_{s=0}
\]
(in more sophisticated terms, we study the infinitesimal action of \(\Gamma\)).
\begin{rema}
  Since the actions of \(G\) and \(\Gamma\) commute (by associativity of the matrix product), the operators \(L_\gamma\) commute with the \(G\)-action.
\end{rema}
Recall that for \(x\in E\) the parametrized curve \(s\mapsto x(s)\bydef \gamma(s)\ast x\) is called the \textsl{flow} of \(L_\gamma\) at \(x\). It is the unique (global) solution to the Cauchy problem
\[
  \left\lbrace
    \begin{aligned}
      x(0)&=x;
      \\
      x'(s)&=L_\gamma(x(s)).
    \end{aligned}
  \right.
\]

We will be particularly interested in the following \(n^2\) subgroups:
\begin{itemize}
  \item
    for \(\alpha\neq\beta\), the \(1\)-parameter subgroup
    \[
      U_{\alpha\beta}
      \bydef
      \Set{I+s E_{\alpha\beta},s\in\R},
    \]
    where we denote by \((E_{\alpha\beta})\) the canonical basis of the vector space of \((m\times m)\)-matrices.
  \item
    for \(\alpha=\beta\), the \(1\)-parameter subgroup
    \[
      T_{\alpha}
      \bydef
      \Set{\diag(1,\dotsc,1,e^s,1,\dotsc,1),s\in \R},
    \]
    where \(e^s\) occupies the \(\alpha\)th diagonal entry.
\end{itemize}
The subgroups \(U_{\alpha\beta}\) with \(\alpha<\beta\) and \(\beta<\alpha\) generate respectively the subgroups of upper and lower unitriangular matrices.
The subgroups \(T_{\alpha}\) generate the subgroup of positive diagonal matrices.
By LU-decomposition, the subgroups of upper unitriangular, lower unitriangular, and positive diagonal matrices generate the connected subgroup
\(
\Gamma^+
=
\Set{\gamma\in\Gamma\colon \det(\gamma)>0}
\).

\begin{exam}
  We take \(E=\R[V^m]\).
  The flow of \(U_{\alpha\beta}\) is given by
  \[
    f
    \mapsto
    f(\vb*{v}_1,\dotsc,\vb*{v}_\beta+s\vb*{v}_\alpha,\dotsc,\vb*{v}_m).
  \]
  Differentiating in order to compute \(L_\gamma\), we retrieve the polarization operator \(L_\gamma=\sigma_{\alpha\beta}\).
  The flow of \(T_\alpha\) is given by
  \[
    f
    \mapsto
    f(\vb*{v}_1,\dotsc,e^s \vb*{v}_\alpha,\dotsc,\vb*{v}_m).
  \]
  Differentiating in order to compute \(L_\gamma\), we retrieve the scaling operator \(L_\gamma=\sigma_{\alpha\alpha}\).
\end{exam}

We will now link the concepts of polarization and of subrepresentation.
\begin{theo}
  Let \(E\) be a rational \(\Gamma\)-representation.
  A subspace \(F\) of \(E\) is stable under the \(\Gamma\)-action if and only if it is stable under the infinitesimal \(\Gamma\)-action.
\end{theo}
\begin{proof}
  If \(F\) is stable under the \(\Gamma\)-action, it is clearly stable under the infinitesimal \(\Gamma\)-action, because \(F\) is closed in \(E\).

  If \(F\) is stable under the infinitesimal \(\Gamma\)-action, the flow of the infinitesimal action preserves \(F\) as well, because it is the solution of a Cauchy problem in \(F\).
  Therefore \(F\) is stable under \(\Gamma^+\).
  We can assume that there is a basis \((\vb*{e}_1,\dotsc,\vb*{e}_n)\) of \(E\) such that \(F\) is spanned by \((\vb*{e}_1,\dotsc,\vb*{e}_r)\). 
  Let \(\rho\colon\Gamma\to\End(E)\) be the linear action on \(E\).
  For any \(\gamma\in\Gamma^+\) and any \(x\in F\):
  \[
    \forall i>r\qcomma
    (\rho(\gamma,x))_i=0.
  \]
  This is a polynomial vanishing on a non-empty open set of \(\Gamma\) [up to clearing the denominator].
  But a polynomial vanishing on a non-empty open set is zero, because its Taylor expansion is zero in any point of the open set.
  Hence \(F\) is \(\Gamma\)-stable.
\end{proof}

As a special case, we obtain that a subspace of \(\R[V^m]\) is a sub-\(\Gamma\)-representation if and only if it is stable under the action of all polarization operators \(\sigma_{\alpha\beta}\) (possibly with \(\alpha=\beta\)).
For a subset \(S\in\R[V^m]\), we define the \textsl{polarization} of \(S\) to be the smaller vector subspace containing \(S\) and stable under the action of all polarization operators. In other words it the smaller \(\Gamma\)-subrepresentation \(\langle\Gamma\ast S\rangle\) containing \(S\).

\begin{rema}
  Note that Theorem~\ref{theo:pure_powers} states that, in one vector variable (\(m=1\))
  \[
    \R[V]_d
    =
    \langle G\ast \ell^d\rangle,
  \]
  for \(G=\GL(V)\) and \(\ell\) any non-zero linear form.
  This is another ``polarization'' result (cf. the ``polarization formulas'' of Example~\ref{exam:pure_powers}).
  It is another similar but different concept that the one that we have described in this section. 
  For \(m\geq 1\) there is actually two possible polarizations. The group \(G\) acts on components (lines of \(\vb{M}\)), and the group \(\Gamma\) acts on vectors (columns of \(\vb{M}\)).
  In this text, ``polarization'' will \emph{never} refer to an action of \(G=\GL(V)\), but always to an action of \(\Gamma=\GL(m)\).
\end{rema}

\part{Polynomial invariants.}
\section{Weyl's Polarization Theorem.}
As outlined in the previous section, Weyl’s approach reformulate Classical Invariant Theory in terms of multilinear algebra using a natural action of the general linear group \(\Gamma=\GL(m)\).
By associativity of the matrix product, the action \(\rho\colon G\to\GL(V)\) and the action of \(\Gamma=\GL(m)\) on \(V^m\) commute, and the same holds for the induced actions on  the coordinate ring \(\R[V^m]\). Hence the invariant subring \(\R[V^m]^G\) is also a \(\Gamma\)-representation.

For \(p<m\), there is an injective homomorphism \(\GL(p)\hookrightarrow\GL(m)\), obtained by taking the identity on the standard supplement of \(\R^p\) in \(\R^m\). Hence, every \(\GL(m)\)-representation \(\GL(m)\to\GL(W)\) can be seen by restriction as a representation \(\GL(p)\to\GL(W)\).
There is also a standard injection \(V^p\hookrightarrow V^m\) (by completing vectors with zeros), and it commutes with the \(\GL(V)\)-actions.
One has therefore
\begin{equation}
  \label{eq:restriction}
  \R[V^p]^{G}
  =
  \R[V^p]\cap \R[V^m]^{G},
\end{equation}
as \(\GL(p)\)-representations.
In other words, a polynomial \(G\)-invariant of \(V^p\) is a polynomial \(G\)-invariant of \(V^m\) that depends only on the \(p\) first vector variables.
For any \(p < m\), the ring of \(p\)-ary \(G\)-invariant forms can thus be recovered from the ring of \(m\)-ary \(G\)-invariant forms.

Since \(\R[V^m]^G\) is a \(\Gamma\)-representation, one has even
\begin{equation}
  \label{eq:pol}
  \langle\Gamma\ast\R[V^p]^{G}\rangle
  \subseteq
  \R[V^m]^{G}.
\end{equation}
In an attempt to construct a tractable basis of invariants forms, one naturally wonders:
\begin{quote}
  \itshape
  How many variables (if finite) are needed to capture all irreducible invariants? In other words, does \eqref{eq:pol} becomes an equality for large \(p\), and if yes for what value of \(p\)? 
\end{quote}

A fundamental result of Weyl, shows that \eqref{eq:pol} is an equality for \(p\geq\dim(V)\).
In particular,
\[
  \R[V^m]^{G}
  =
  \langle\Gamma\ast\R[V^n]^{G}\rangle
  \qcomma\qfor m\geq n.
\]
The main ingredient of the proof is column-wise Gaussian elimination.
We denote by \(U_m\) the subgroup of the general linear group \(\Gamma=\GL(m)\) consisting of upper triangular matrices with ones on the diagonal (\textsl{upper unitriangular} matrices).
It is precisely the group of column-wise Gaussian elimination.

An \(m\)-ary form \(f\in\R[V^m]\) that is simultaneously invariant under the diagonal action of \(G\) on \(V^m\) and under the action of \(U_m\subset\Gamma\) will be called a \textsl{bi-invariant}.
Note that since the actions of \(G\) and \(\Gamma\) commute, one has
\[
  \R[V^m]^{G\times U_m}
  =
  (\R[V^m]^{G})^{U_m}
  =
  (\R[V^m]^{U_m})^{G}.
\]

\begin{THEO}
  [Weyl's Polarization Theorem]
  \label{theo:polarization}
  Let \(V\) be a \(n\)-dimensional rational representation of a linear group \(G\).
  \begin{enumerate}
    \item
      All \(G\)-invariant \(m\)-ary forms can be recovered from \((G\times U_m)\)-invariant \(m\)-ary forms by polarization:
      \[
	\R[V^m]^G
	=
	\langle\Gamma\ast\R[V^m]^{(G\times U_m)}\rangle.
      \]
    \item
      For any \(m\), the \(U_m\)-invariant \(m\)-ary forms can be obtained from \(U_n\)-invariant \(n\)-ary forms as 
      \[
	\R[V^m]^{U_m}
	=
	\R[V^m]
	\cap
	\R[V^n]^{U_n}.
      \]
    \item
      As a consequence, the \(G\)-invariant \(m\)-ary forms can be recovered from the \(G\)-invariant \(n\)-ary forms by polarization if \(m\geq n\) as
      \[
	\R[V^m]^G
	=
	\langle \Gamma\ast \R[V^n]^G\rangle,
      \]
      and by restriction if \(m\leq n\).
  \end{enumerate}
\end{THEO}
\begin{proof}
  \begin{enumerate}[wide]
    \item
      Any finite-dimensional \(\Gamma=\GL(m)\)-representation \(W\) is generated by its \(U_m\)-invariants:
      \[
	W
	=
	\langle \Gamma\ast W^{U_m}\rangle.
      \]
      This fundamental result may be admitted on a first reading, and its proof is postponed to the appendix.

      Let \(W\bydef\R[V^m]^G\).
      It is not finite-dimensional, but \(W=\bigoplus W_d\), where
      \[
	W_{d}
	\bydef
	\big\langle f\in\R[V^m]^G,\deg(f)=d\big\rangle
      \]
      is finite-dimensional.
      Since the actions of \(G\) and of \(\Gamma\) commute and preserve the degree, \(W_d\) is a \(\Gamma\)-representation.
      We apply the general result above to this \(\Gamma\)-representation in order to obtain
      \[
	W_{d}
	=
	\langle\Gamma\ast{W_{d}}^{U_m}\rangle.
      \]
      Since \(\R[V^m]=W=\bigoplus W_d\), as \((G\times\Gamma)\)-representations, by summing
      \[
	\R[V^m]^G
	=
	\langle\Gamma\ast(\R[V^m]^G)^{U_m}\rangle
	=
	\langle\Gamma\ast\R[V^m]^{(G\times U_m)}\rangle.
      \]
    \item
      For \(m\leq n\), the subspace \(V^m\subseteq V^n\) is \(U_n\)-stable,
      and
      \[
	\R[V^m]^{U_m}
	=
	\R[V^m]^{U_n}
	=
	\R[V^m]\cap\R[V^n]^{U_n}.
      \]

      We now assume that \(m\geq n\).
      Since \(\R[V^n]^{U_n}\subseteq \R[V^m]\), we want to prove that
      \[
	\R[V^m]^{U_m}
	=
	\R[V^m]\cap\R[V^n]^{U_n}
	=
	\R[V^n]^{U_n},
      \]
      as subspaces of \(\R[V^m]\).
      Inverting the roles of \(m\) and of \(n\) in the previous reasoning the right hand side becomes
      \[
	\R[V^n]^{U_n}
	=
	\R[V^n]\cap\R[V^m]^{U_m}.
      \]
      It remains hence to prove that \(\R[V^m]^{U_m}=\R[V^m]^{U_m}\cap\R[V^n]\), in other words that 
      \[
	\R[V^m]^{U_m}\subseteq \R[V^n].
      \]

      An \(m\)-ary form is \(U_m\)-invariant if and only if it is constant on every \(U_m\)-orbit.
      We first restrict ourselves to the open dense subset \(\Omega\) of \(V^m\) consisting of elements \((\vb*{v}_1,\dotsc,\vb*{v}_m)\) such that \(\det(\vb*{v}_1,\dotsc,\vb*{v}_n)\neq 0\), that is the \(n\) first vectors are linearly independent.
      By Gaussian elimination, the point \((\vb*{v}_1,\dotsc,\vb*{v}_n,0,\dotsc,0)\) lies in the \(U_m\)-orbit of \((\vb*{v}_1,\dotsc,\vb*{v}_m)\).
      If \(f\) is a \(U_m\)-invariant \(m\)-ary form, we have that \(f(\vb*{v}_1,\dotsc,\vb*{v}_m)=f(\vb*{v}_1,\dotsc,\vb*{v}_n,0,\dotsc,0)\) on \(\Omega\). Since \(\Omega\) is dense, this polynomial identity is globally true.
      Hence any \(U_m\)-invariant \(m\)-ary form is in the subring \(\R[V^n]\) of \(n\)-ary forms.

    \item
      We are now in position to prove the last point.
      For \(m\leq n\), it is the content of \eqref{eq:restriction}.
      For \(m\geq n\), one infers from the first point that
      \[
	\R[V^m]^G
	=
	\langle \GL(m)\ast(\R[V^m]^{G})^{U_m}\rangle
	=
	\langle \GL(m)\ast(\R[V^m]^{U_m})^G\rangle,
      \]
      and similarly
      \[
	\R[V^n]^G
	=
	\langle \GL(n)\ast(\R[V^n]^{U_n})^G\rangle.
      \]
      From the second point, one infers that
      \[
	S
	\bydef
	(\R[V^n]^{U_n})^G
	=
	(\R[V^m]^{U_m})^G.
      \]
      Since \(\GL(n)\) is a subgroup of \(\GL(m)\), for any set \(S\in\R[V^m]\) one has
      \[
	\big\langle\GL(m)\ast\langle\GL(n)\ast S\rangle\big\rangle
	=
	\langle\GL(m)\ast S\rangle.
      \]
      Therefore
      \[
	\R[V^m]^G
	=
	\langle \GL(m)\ast S\rangle
	=
	\big\langle\GL(m)\ast\langle\GL(n)\ast S\rangle\big\rangle
	=
	\langle \GL(m)\ast \R[V^n]^G\rangle.
	\qedhere
      \]
  \end{enumerate}
\end{proof}

For \(\gamma_i\in\Gamma\), \(P\in\R[X_1,X_2,\dotsc]\), \(s_{ij}\in S\), one has
\[
  \sum_i \gamma_i \ast P_i(s_{ij})
  =
  \sum_i P_i(\gamma_i \ast s_{ij}).
  \qedhere
\]
Therefore, if \(W\) is a subalgebra of \(\R[V^m]\), that is algebraically generated by a set \(S\), then 
\(\langle \Gamma\ast W\rangle\) is algebraically generated by \(\Gamma\ast S\).
Moreover, the action of \(\Gamma\) on \(\R[V^m]\) preserves the degree.
One infers that \textit{if the \(G\)-invariant \(n\)-ary forms are algebraically generated by elements of degree \(\leq d\),
then it is the case for \(G\)-invariant forms in any number of variables.}

Theorem~\ref{theo:polarization} can be generalized in order to deal with joint invariant forms.
\begin{THEO'}{theo:polarization}
  Let \(V_1,\dotsc,V_k\) be rational finite-dimensional representations of dimensions \(n_1,\dotsc,n_k\) of a linear group \(G\).
  Then the joint invariant ring of \((m_1,\dotsc,m_k)\)-ary forms can be recovered by polarization and restriction from the \((G\times U_{n_1}\times\dotsb\times U_{n_k})\)-invariant \((n_1,\dotsc,n_k)\)-ary forms.
\end{THEO'}

\begin{rema}
  The polarization theorem applies to \emph{every} linear group.
  For instance every symmetric group (seen as a group of permutation matrices) is a linear group and more generally every finite group is a linear group (by the theorem of Cayley).
  In general linear subgroups can be quite wild: not necessarily ``reductive'', etc\ldots
\end{rema}

\section{Invariant forms for the general and special linear groups.}
The case of the standard representation \(V=\R^n\) of the group \(G=\GL(n)\) is not very interesting. Picking a basis, we identify \(V^n\) with \(\operatorname{M}_n\). The orbit of the matrix \(I\) is the open dense set \(\GL(n)\).
Hence every polynomial invariant on \(V^n\) is constant.

We will rather work out the example of the representation \(V^n\oplus(V^*)^n\). 
According to Theorem~\ref{theo:polarization}', we need to understand the \((\GL(n)\times U_n\times U_n)\)-invariants forms, which we call \textsl{tri-invariant forms}.

We form the matrix \(\vb{M}\) which columns are the coordinates of the vectors \(\vb*{v}_1,\dotsc,\vb*{v}_n\) and the matrix \(\vb{A}=\Mat(f^1,\dotsc,f^n)\) which rows are the coordinates of the covectors \(f^1,\dotsc,f^n\).
Then we consider the standard diagonal action 
\[
  g\ast(\vb{M},\vb{A})=(g\vb{M},\vb{A}g^{-1})
\]
of the group \(G=\GL(n)\), the standard columnwise action 
\[
  \gamma\ast \vb{M}=\vb{M}\gamma^{-1}
\]
of the group \(\Gamma=\GL(\dim(V))=\GL(n)\), and the ``flipped'' rowwise action 
\[
  \gamma'\ast \vb{A}=J\gamma' J\vb{A}
\]
of the group \(\Gamma'=\GL(\dim(V^*))=\GL(n)\), where \(J\) is the exchange matrix, i.e. the permutation matrix with ones on the antidiagonal.
We make this last choice so that the upper unitriangular group \(U_n\in\GL(n)\) actually acts on \(\vb{A}\) by multiplication with lower unitriangular matrices.
These three actions commute.

We will soon restrict to an open dense set of \(V^n\oplus(V^*)^n\), which will allow us to define a normal form in every \((\GL(n)\times U_n\times U_n)\)-orbit.
Recall that the leading minor \(\Delta_i(\vb{M})\) of a square matrix \(\vb{M}\) is the determinant of the submatrix formed of the \(i\) first rows and of the \(i\) first columns of \(\vb{M}\).
\begin{theo}[LU decomposition]
  Let \(\vb{M}\in\GL(n)\) be a matrix, all the leading minors of which are non-zero.
  Then there exists a unique decomposition \(\vb{M}=\gamma'D\gamma\) where \(\gamma'\) is lower unitriangular, \(\gamma\) is upper unitriangular, and \(D\) is diagonal.
  Moreover the diagonal coefficient \(\lambda_i\) of \(D\) is the ratio \(\Delta_i(\vb{M})/\Delta_{i-1}(\vb{M})\).
\end{theo}
\begin{proof}
  We apply Gaussian elimination (without row exchanges) to reduce \(\vb{M}\) to an upper-triangular matrix, then we factorize the diagonal.
  Uniqueness follows from the observation that the only matrices that are both upper and lower triangular are diagonal.
\end{proof}

\begin{theo}\label{theo:triinv}
  The tri-invariant forms on \(V^n\oplus(V^*)^n\)
  under the action of \(\GL(n)\)
  are generated by the determinants
  \[
    D_k(\vb*{v}_1,\dotsc,\vb*{v}_n;f^1,\dotsc,f^n)
    \bydef
    \begin{vmatrix}
      f^1(\vb*{v}_1)&\dotso&f^1(\vb*{v}_k)\\
      \\
      f^k(\vb*{v}_1)&\dotso&f^k(\vb*{v}_k)
    \end{vmatrix}.
  \]
  There is thus a basis of \(\GL(n)\)-invariant of degree at most \(2n\).
\end{theo}
\begin{proof}
  Let \(P\) be a tri-invariant form.
  For any \(g\in\GL(n)\), any \(\gamma\in U_n\) and any \(\gamma'\in JU_nJ\), one has
  \[
    P(\vb{M},\vb{A})
    =
    P(g\vb{M}\gamma,\gamma' \vb{A}g^{-1}).
  \]
  We restrict to the open dense subset \(\Omega\) where \(\vb{A}\) is invertible and where all the leading minors of \(\vb{A}\vb{M}\) are non-zero.
  Then \(\vb{A}\vb{M}\) admits a LU decomposition \(\vb{A}\vb{M}=(\gamma')^{-1}D\gamma^{-1}\). Taking \(g=\gamma'\vb{A}\), one has:
  \[
    P(\vb{M},\vb{A})
    =
    P(\gamma'\vb{A}\vb{M}\gamma,I)
    =
    P(D,I).
  \]

  Since for every \(i\), \(\lambda_i=D_{i}/D_{i-1}\) is a rational bi-invariant, we can treat every monomial of \(P\) separately.
  We admit the algebraic independence of the leading minors \(D_k=\Delta_k(\vb{A}\vb{M})\) (their dependance on the parameters is triangular).
  The monomial 
  \[
    \lambda_1^{\alpha_1}\dotsm\lambda_n^{\alpha_n} = D_{1}^{\alpha_1-\alpha_2}\dotsm D_{m-1}^{\alpha_{m-1}-\alpha_m}D_m^{\alpha_m}
  \]
  is polynomial if and only if \(\alpha_1\geq\dotsb\geq\alpha_n\).

  Any tri-invariant form on \(\Omega\) is thus a polynomial in leading minors of \(\vb{A}\vb{M}\).
  This is globally true because a polynomial identity that holds on on a dense open subset holds on \(V^n\oplus (V^*)^n\).
\end{proof}

\begin{coro}\label{coro:SW}
  Every polynomial \(\GL(n)\)-invariant of vectors and covectors can be obtained by algebraic combinations and polarization from the duality bracket \(V\oplus V^*\to\R\).
\end{coro}

\begin{theo}
  For \(n\geq 2\), the tri-invariant forms on \(V^n\oplus(V^*)^n\) under the standard action of \(\SL(n)\) are generated by \(D_1,\dotsc,D_{n-1}\) and the determinants.
  There is thus a basis of \(\SL(n)\)-invariant of degree at most \(2n-2\).
\end{theo}
\begin{proof}
  We keep the notation of the previous proof.
  The normal form of \((\vb{M},\vb{A})\in\Omega\) in its \((\SL(n)\times U_n\times U_n)\)-orbit is \((\det(\vb{A})^{-1}D,\det(\vb{A}))\).
  Therefore, any tri-invariant form is a Laurent polynomial in \(D_1,\dotsc,D_{n},\det(\vb{A})\).
  One has \(D_n(\vb{M},\vb{A})=\det(\vb{A})\det(\vb{M})\), so we can replace \(D_n\) by \(\det(\vb{M})\).
  We conclude similarly as above by algebraic independence of \(D_1,D_2,\dotsc,D_{n-1},\det(\vb{A}),\det(\vb{M})\).
\end{proof}

\begin{coro}\label{coro:SLSW}
  Every polynomial \(\SL(n)\)-invariant of vectors and covectors can be obtained by algebraic combinations and polarization from the duality bracket \(V\oplus V^*\to\R\), the determinant of vectors \(V^n\to\R\) and the determinant of covectors \(V^{*n}\to\R\).
\end{coro}

\section{Invariant forms for orthogonal groups.}
Next, we work out the example of the standard representation \(V=\R^n\) of the group \(G=\GO(n)\).
We identify \(V^n\) with the vector space of square matrices \(\operatorname{M}_n\).
We can restrict to the open dense set \(\GL(n)\), which allows us to define a normal form in every \((\GO(n)\times U_n)\)-orbit.

Recall that for a rectangle matrix \(\vb{M}=\Mat(\vb*{v}_1,\dotsc,\vb*{v}_k)\), the Gram determinant \(\det(\vb{M}^*\vb{M})\) is the square of the volume of the parallelepiped formed by the vectors \(\vb*{v}_1,\dotsc,\vb*{v}_k\) in \(V\). This is known as the \textsl{Gram identity}, a slight refinement of the identity \(\det(\vb{M}^*\vb{M})=\det(\vb{M})^2\) for square matrices. We denote this volume by \(\lVert \vb*{v}_1\wedge\dotsb\wedge \vb*{v}_k\rVert\).

\begin{theo}[QR decomposition]
  Any real square matrix \(\vb{M}=\Mat(\vb*{v}_1,\dotsc,\vb*{v}_n)\in\GL(n)\) decomposes uniquely as \(\vb{M}=gD\gamma^{-1}\), where \(g\in\GO(n)\), \(\gamma\in U_n\) and \(D\) is positive diagonal.
  Moreover the diagonal coefficients \(\lambda_j\) of \(D\) satisfy
  \[
    \lVert \vb*{v}_1\wedge\dotsb\wedge \vb*{v}_j\rVert
    =
    \lambda_1\dotsm\lambda_j.
  \]
\end{theo}
\begin{proof}
  The Gram--Schmidt process yields a unique upper-triangular change of basis \(P\) with positive diagonal such that 
  \(g\bydef \vb{M}P\) is orthogonal. We write \(P^{-1}=D\gamma^{-1}\), for \(\gamma\in U_n\) and \(D\) positive diagonal,
  with coefficients \(\lambda_1,\dotsc,\lambda_n\).

  Denote by \(\vb*{v}_1,\dotsc,\vb*{v}_n\) the columns of \(\vb{M}\) and denote by \(\vb*{w}_1,\dotsc,\vb*{w}_n\) the columns of \(P^{-1}\).
  Then for \(i,j\), one has 
  \[
    \langle \vb*{v}_i,\vb*{v}_j\rangle=\langle g\vb*{w}_i,g\vb*{w}_j\rangle=\langle \vb*{w}_i,\vb*{w}_j\rangle.
  \]
  Accordingly
  \[
    \det(\Mat(\vb*{v}_1,\dotsc,\vb*{v}_j)^*\Mat(\vb*{v}_1,\dotsc,\vb*{v}_j))
    =
    \det(\Mat(\vb*{w}_1,\dotsc,\vb*{w}_j)^*\Mat(\vb*{w}_1,\dotsc,\vb*{w}_j))
    =
    (\lambda_1\dotsm\lambda_j)^2.
  \]
  The last equality holds for all \(j\) because \(\gamma\in U_n\) is unitriangular.
\end{proof}

\begin{theo}
  The bi-invariant \(n\)-ary forms on the standard representation of \(\GO(n)\) are generated by the Gram determinants
  \[
    G_k(\vb*{v}_1,\dotsc,\vb*{v}_n)
    \bydef
    \begin{vmatrix}
      \langle \vb*{v}_1,\vb*{v}_1\rangle&\dotso&\langle \vb*{v}_1,\vb*{v}_k\rangle\\
      \\
      \langle \vb*{v}_k,\vb*{v}_1\rangle&\dotso&\langle \vb*{v}_k,\vb*{v}_k\rangle
    \end{vmatrix}.
  \]
  There is thus a basis of \(\GO(n)\)-invariant of degree at most \(2n\).
\end{theo}
\begin{proof}
  Let \(P\in\R[V^n]\) be a bi-invariant \(n\)-ary form.
  For an invertible matrix \(\vb{M}=\Mat(\vb*{v}_1,\dotsc,\vb*{v}_n)\in\GL(n)\),
  since bi-invariant forms are constant on orbits, one obtains that
  \[
    P(\vb{M})
    =
    P(\lambda_1\vb*{e}_1,\dotsc,\lambda_n\vb*{e}_n),
  \]
  with the notation of the above lemma for the normal form of \(\vb{M}\).
  Now the hyperplane reflections are elements of \(\GO(n)\), hence
  \[
    P(\lambda_1\vb*{e}_1,\dotsc,\lambda_i\vb*{e}_i,\dotsc,\lambda_n\vb*{e}_n)
    =
    P(\lambda_1\vb*{e}_1,\dotsc,-\lambda_i\vb*{e}_i,\dotsc,\lambda_n\vb*{e}_n).
  \]
  One infers that \(P\) depends polynomially on \(\lambda_1^2,\dotsc,\lambda_n^2\).
  Since for every \(i\), \(\lambda_i^2=G_i/G_{i-1}\) is a rational bi-invariant, we can treat every monomial of \(P\) separately.
  We admit the algebraic independence of the Gram determinants \(G_k\) (their dependance on the vectors \(\vb*{v}_i\) is triangular).
  The monomial 
  \[
    \lambda_1^{2\alpha_1}\dotsm\lambda_n^{2\alpha_n} = G_{1}^{\alpha_1-\alpha_2}\dotsm G_{n-1}^{\alpha_{n-1}-\alpha_n}G_n^{\alpha_n}
  \]
  is polynomial in \(v\) if and only if \(\alpha_1\geq\dotsb\geq\alpha_n\).

  Any bi-invariant form is thus a polynomial in the Gram determinants of \(\vb{M}\in\GL(n)\).
  This is globally true on \(V^n\) because a polynomial identity that holds on \(\GL(n)\) holds on \(V^n\), by density.
\end{proof}

\begin{coro}\label{coro:OSW}
  Every polynomial \(\GO(n)\)-invariant of vectors can be obtained by algebraic combinations and polarization from the metric \(V\oplus V\to\R\).
\end{coro}

\begin{rema}
  This proof also works e.g. for the Lorentz group, replacing the Gram--Schmidt process with the Lorentzian Gram--Schmidt process and working on the open dense subset where all leading minors of the Gram matrix \(\vb{M}^*\eta \vb{M}\) are non-zero.
\end{rema}

\begin{theo}
  For \(n\geq 2\), the bi-invariant \(n\)-ary forms on the standard representation of \(\SO(n)\) are generated by the Gram determinants \(G_1,\dotsc,G_{n-1}\) and the determinant.
  There is thus a basis of \(\SO(n)\)-invariant of degree at most \(2n-2\).
\end{theo}
\begin{proof}
  The proof proceeds as before.
  Let \(P\in\R[V^n]\) be a bi-invariant \(n\)-ary form.
  For an invertible matrix \(\vb{M}=\Mat(\vb*{v}_1,\dotsc,\vb*{v}_n)\in\GL(n)\),
  since bi-invariant forms are constant on orbits, one obtains that
  \[
    P(\vb{M})
    =
    P(\pm\lambda_1\vb*{e}_1,\lambda_2\vb*{e}_2,\dotsc,\lambda_n\vb*{e}_n),
  \]
  with the notation of the above lemma for the normal form of \(\vb{M}\).
  Here the sign \(\pm\) is the sign of \(\det(\vb{M})\), because we want \(g\in\SO(n)\).
  Now the central symmetries in coordinate planes are elements of \(\GO(n)\), hence
  \[
    P(\pm\lambda_1\vb*{e}_1,\dotsc,\lambda_i\vb*{e}_i,\dotsc,\lambda_j\vb*{e}_j,\dotsc,\lambda_n\vb*{e}_n)
    =
    P(\pm\lambda_1\vb*{e}_1,\dotsc,-\lambda_i\vb*{e}_i,\dotsc,-\lambda_j\vb*{e}_j,\dotsc,\lambda_n\vb*{e}_n).
  \]
  One infers that \(P\) has only monomials where all \(\lambda_i\) have the same parity.
  Therefore
  \[
    P(\vb{M})
    =
    \pm(\lambda_1\dotsm\lambda_n)
    f(\lambda_1^2,\dotsc,\lambda_n^2)
    +
    g(\lambda_1^2,\dotsc,\lambda_n^2).
  \]
  As before, we deduce that
  \[
    P(\vb{M})
    =
    \pm(\lambda_1\dotsm\lambda_n)
    \tilde f(G_1,\dotsc,G_n)
    +
    \tilde g(G_1,\dotsc,G_n).
  \]
  Lastly we notice that \(\pm(\lambda_1\dotsm\lambda_n)=\det(\vb{M})\), and \(G_n(\vb{M})=\det(\vb{M})^2\).

  Any bi-invariant form is thus a polynomial in the Gram determinants \(G_1,\dotsc,G_{n-1}\) and the determinant of \(\vb{M}\in\GL(n)\).
  This is globally true on \(V^n\) because a polynomial identity that holds on \(\GL(n)\) holds on \(V^n\), by density.
\end{proof}

\begin{coro}\label{coro:SOSW}
  Every polynomial \(\SO(n)\)-invariant of vectors can be obtained by algebraic combinations and polarization from the metric \(V\oplus V\to\R\) and the determinant \(V^{n}\to \R\).
\end{coro}

\part{Tensor invariants.}\label{part:tensor}
\section{Weyl's First Fundamental Theorem.}
We now truly enter Weyl's philosophy of linearizing Classical Invariant Theory.
\textit{All polynomial invariants of classical groups are symmetrizations of tensor expressions built only from:
tensor products, permutations of indices, contractions using the fundamental invariant bilinear forms and (for special groups) determinants.}
For a finite-dimensional vector space \(V\), we denote the \textsl{tensor algebra} of \(V\) by
\(
T^{\bullet}(V)
\bydef
\bigoplus_{p\geq 0} V^{\otimes p}
\).
Its \textsl{mixed tensor algebra} is then the product \(T^\bullet(V)\otimes T^\bullet(V^*)\).
The following general result is admitted.
\begin{THEO}[Weyl's typical invariants]
  Let \(V\) be a finite-dimensional rational representation of a classical group \(G\). There exists a finite number of basic invariants that generate the mixed tensor algebra \((T^{\bullet}(V)\otimes T^{\bullet}(V^*))^G\).
\end{THEO}
The following illustrative examples will be worked below.
\begin{itemize}
  \item
    The only basic tensor invariant for the standard representation of the general linear group \(\GL(n)\) is the duality bracket;
  \item
    For the special linear group \(\SL(n)\), there is also the determinant.
  \item
    The only basic tensor invariants for the standard representation of the orthogonal group \(\GO(n)\) are the metric and the cometric;
  \item
    For the special orthogonal group \(\SO(n)\), there is also the determinant.
  \item
    The only basic tensor invariants for the conjugation action on \(\End(V)\) of the general linear group \(\GL(V)\) are the trace and the matrix product.
  \item
    The only basic tensor invariants for the conjugation action on \(\End(V)\) of the orthogonal group \(\GO(V)\) are the trace, the transpose and the matrix product.
\end{itemize}

\section{Tensor invariants for general and special linear groups.}
Let \(V\) be a \(n\)-dimensional vector space.
We consider the invariants of the mixed tensor algebra
\[T^\bullet(V)\otimes T^\bullet(V^*)\]
under the diagonal action of \(\GL(V)\).
This is a bigraded algebra, with graded pieces
\[
  T^{p,q}(V)
  =
  V^{\otimes p}\otimes V^{*\otimes q}
\]
and the action of \(\GL(V)\) respects this bigrading.

\begin{theo}
  The only basic \(\GL(n)\)-invariant tensor of the standard representation \(V\) is the duality pairing \(V^*\otimes V\to\R\).
  Every \(\GL(n)\)-invariant mixed tensor can be built from this pairing and tensor index permutations,
  or equivalently from index contractions:
  \[
    (\vb*{v}_1,\dotsc,\vb*{v}_m;f^1,\dotsc,f^m)
    \mapsto
    f^{i}(\vb*{v}_{j}).
  \]
\end{theo}
\begin{proof}
  Invariants tensors \(T^{p,q}(V)\to\R\) naturally identify with invariant multilinear polynomials in \(\R[V^p\oplus V^{*q}]\).
  The result hence follows from Corollary~\ref{coro:SW}.
\end{proof}

\begin{theo}
  The only basic \(\SL(n)\)-invariant tensors of the standard representation \(V\) are the duality pairing \(V\otimes V^*\to\R\), and the determinant \(V^{\otimes n}\to\R\).
  Every \(\SL(n)\)-invariant mixed tensor can be built from these basic tensors and tensor index permutations,
\end{theo}
\begin{proof}
  This time it folllows from Corollary~\ref{coro:SLSW}.
  Note that the determinant of covectors can be recovered from the determinant and the inverse of the duality pairing.
\end{proof}

\section{Tensor invariants for orthogonal groups.}
Let \(V\) be a \(n\)-dimensional vector space.
We consider the invariants of the mixed tensor algebra \(T^\bullet(V)\otimes T^\bullet(V^*)\) under the diagonal action of \(\GO(V)\).

\begin{theo}
  The only basic \(\GO(n)\)-invariant tensors of the standard representation \(V\) are the metric \(V\otimes V\to\R\), and the cometric \(\R\to V\otimes V\).
  Every \(\GO(n)\)-invariant mixed tensor can be built from these pairings and tensor index permutations.
\end{theo}
\begin{proof}
  Since the euclidean metric \(g \colon V \to V^*\) and the euclidean cometric \(g^{-1}\colon V^*\to V\)  are a pair of inverse \(\GO(n)\)-invariant tensors, we can identify \(V\) and \(V^*\) as needed in order to reduce the study of the tensor algebra \(T^\bullet(V)\). Then the results follows from Corollary~\ref{coro:OSW}.
\end{proof}

\begin{theo}
  The only basic \(\SO(n)\)-invariant tensors of the standard representation \(V\) are the metric \(V\otimes V\to\R\), the cometric \(\R\to V\otimes V\), and the determinant \(V^{\otimes n}\to\R\).
  Every \(\SO(n)\)-invariant mixed tensor can be built from these basic tensors and tensor index permutations.
\end{theo}
\begin{proof}
  Since the euclidean metric \(g \colon V \to V^*\) and the euclidean cometric \(g^{-1}\colon V^*\to V\)  are a pair of inverse \(\GO(n)\)-invariant tensors, we can identify \(V\) and \(V^*\) as needed in order to reduce the study of the tensor algebra \(T^\bullet(V)\). Then the results follows from Corollary~\ref{coro:SOSW}.
\end{proof}

\section{Tensor invariants on matrix spaces.}
We consider the \(\GL(V)\)-action on \(\End(V)\) by conjugation.
Since \(\End(V)\simeq V\otimes V^*\), one has \(\End(V)=\End(V)^*\).
Therefore, we consider the tensor algebra \(T^\bullet(\End(V))\).
The invariant tensors in \(\End(V)^{\otimes m}=\End(V^{\otimes m})\) are the endomorphisms \(V^{\otimes m}\to V^{\otimes m}\) that commute with the \(\GL(V)\)-action.
The group of permutations \(S_m\) acts linearly on \(V^{\otimes m}\) by tensor index permutations
\[
  \sigma\ast(\vb*{v}_1\otimes\dotsb\otimes \vb*{v}_m)
  =
  (\vb*{v}_{\sigma^{-1}(1)}\otimes\dotsb\otimes \vb*{v}_{\sigma^{-1}(m)}).
\]
The \textsl{group algebra} \(\R[S_m]\) is the vector space generated by permutations, equipped with the linear extension of the group product. 
Since the action of \(S_m\) commutes with the diagonal \(\GL(V)\)-action, one has
\[
  \R[S_m]
  \subseteq
  (\End(V^{\otimes m}))^{\GL(V)}.
\]
\textsl{Schur--Weyl duality} asserts that every \(\GL(V)\)-equivariant endomorphism of \(V^{\otimes m}\) actually arises from this action.

\begin{theo}
  The algebra of endomorphisms of \(V^{\otimes m}\) commuting with the \(\GL(V)\)-action is naturally isomorphic with the group algebra \(\R[S_m]\) of the symmetric group acting by permutation of the tensor factors.
\end{theo}
\begin{proof}
  The representation \(\End(V^{\otimes m})\), with the action of \(\GL(V)\) by conjugation identifies with the representation \(V^m\otimes V^{*m}\), with the standard diagonal action of \(\GL(V)\).
  But we know that any invariant of \(V^m\otimes V^{*m}\) is built from contractions.
  Hence the basic invariants are complete contractions
  \[
    T^{\sigma}
    \colon
    (\vb*{v}_1,\dotsc,\vb*{v}_m;f^1,\dotsc,f^m)
    \mapsto
    f^{\sigma(1)}(\vb*{v}_1)\dotsm f^{\sigma(m)}(\vb*{v}_m).
  \]
  One has
  \[
    T^\sigma(\vb*{e}_{i_1},\dotsc,\vb*{e}_{i_m};\vb*{e}^{j_1},\dotsc,\vb*{e}^{j_m})
    =
    \delta_{i_1}^{\sigma(j_{1})}
    \dotsm
    \delta_{i_m}^{\sigma(j_{m})}
    =
    \vb*{e}^{j_1}\otimes\dotsb\otimes\vb*{e}^{j_m}
    \big(\sigma\ast(\vb*{e}_{i_1}\otimes\dotsb\otimes\vb*{e}_{i_m})\big).
  \]
  Therefore, \(T_\sigma\) and \(\sigma\ast\) are identified through the isomorphism \(V^m\otimes V^{*m}=\End(V^{\otimes m})\).
\end{proof}

\begin{theo}
  The only basic \(\GL(n)\)-invariant mixed tensors of the action on \(\End(V)\) by conjugation are the trace \(\End(V)\to\R\) and the matrix product \(\End(V)\otimes\End(V)\to\End(V)\).
  Every \(\GL(n)\)-invariant tensor can be built from these basic tensors and tensor index permutations.
\end{theo}
\begin{proof}
  Any permutation \(\sigma\) can be written uniquely as a product of disjoint cycles \(\tau_1,\dotsc,\tau_k\).
  For each cycle \(\tau=(a,b,\dotsc,z)\) (even for \(1\)-cycles), take the product of matrices 
  \[
    P_\tau
    =
    \vb{A}_a \vb{A}_b\dotsm \vb{A}_z.
  \]
  Then
  \[
    T^\sigma(\vb{A}_1,\dotsc,\vb{A}_m)
    =
    \Trace(P_{\tau_{1}})
    \dotsm
    \Trace(P_{\tau_{k}}).
  \]
  See the examples below for a clearer illustration.
\end{proof}
\begin{exam}
  The \(2\) basic invariants tensors of type \((2,0)\) are:
  \[
    T^{\id}(\vb{A}_1\otimes \vb{A}_2)
    =
    \delta_{i_1}^{j_1}\delta_{i_2}^{j_2}(\vb{A}_1)^{i_1}_{j_1} (\vb{A}_2)_{i_2}^{j_2}
    =
    \Trace(\vb{A}_1)\Trace(\vb{A}_2)
  \]
  and
  \[
    T^{(1,2)}(\vb{A}_1\otimes \vb{A}_2)
    =
    \delta_{i_1}^{j_2}\delta_{i_2}^{j_1}(\vb{A}_1)^{i_1}_{j_1} (\vb{A}_2)_{i_2}^{j_2}
    =
    \Trace(\vb{A}_1\vb{A}_2).
  \]
  The \(6\) basic invariants tensors of type \((3,0)\) are:
  \[
    \{\Trace(A_1)\Trace(A_2)\Trace(A_3);\Trace(A_i)\Trace(A_jA_k);\Trace(A_iA_jA_k)\}.
  \]
\end{exam}

\begin{theo}
  The only basic \(\GO(n)\)-invariant mixed tensors of the action on \(\End(V)\) by conjugation are the trace \(\End(V)\to\R\), the transpose \(\End(V)\to\End(V)\) and the matrix product \(\End(V)\otimes\End(V)\to\End(V)\).
  Every \(\GO(n)\)-invariant tensor can be built from these basic tensors and tensor index permutations.
\end{theo}
\begin{proof}
  The basic invariants are complete contractions with the metric and the cometric.
  So product of symbols \(\delta_{i,i'}\) and \(\delta^{j,j'}\) such that the non-repeating indices are \(\{i_1,\dotsc,i_m\}\) in the lower part and \(\{j_1,\dotsc,j_m\}\) in the upper part.
  In Einstein notation, one has \(\delta_{ik}\delta^{kj}=\delta_{i}^{j}\), so we can eliminate all repeating indices.
  We end up with a symbol which is in bijection with \(2\)-partitions \(\pi\) of the set of indices \(\Set{i_1,\dotsc,i_m,j_1,\dotsc,j_m}\).
  A loop is obtained as follows: starting from an index, follow the pairing, and then go to the lower/upper index of the same factor, and continue this alternation until you return from the starting index. Every \(2\)-partition of the set of indices decomposes in a union of disjoint loop \(\tau_1,\dotsc,\tau_k\).
  For each loop \(\tau\), form a product \(P_\tau\) in the order in which the indices \(a,b,\dotsc,z\) appear, with \(\vb{A}_x\) if the loop goes from \(i_x\) to \(j_x\) and with \(\vb{A}_x^*\) if the loop goes from \(j_x\) to \(i_x\).
  Then
  \[
    T^\pi(\vb{A}_1,\dotsc,\vb{A}_m)
    =
    \Trace(P_{\tau_1})
    \dotsm
    \Trace(P_{\tau_k}).
  \]
  See some basic examples below.
\end{proof}

\begin{exam}
  The \(3\) basic \(\GO(n)\)-invariants tensors \(\End(V)^{\otimes 2}\to\R\) are
  \[
    \{\Trace(\vb{A}_1)\Trace(\vb{A}_2), \Trace(\vb{A}_1\vb{A}_2), \Trace(\vb{A}_1\vb{A}_2^*)\}.
  \]
  Corresponding respectively to
  \[
    \delta_{i_1}^{j_1}\delta_{i_2}^{j_2},\quad
    \delta_{i_1}^{j_2}\delta_{i_2}^{j_1}\quad\text{and}\quad
    \delta_{i_1i_2}\delta^{j_1j_2}.
  \]
  It is in bijection with the set of \(2\)-partitions of \(\Set{i_1,i_2,j_1,j_2}\).
  The graphs of the \(2\)-partitions are:
  \[
    \begin{tikzpicture}
      [>=stealth,pairing/.style={thick,->}, jump/.style={dashed,->}]
      \begin{scope}
	\node (i1) at (0,0) {\(i_1\)};
	\node (i2) at (2,0) {\(i_2\)};
	\node (j1) at (0,2) {\(j_1\)};
	\node (j2) at (2,2) {\(j_2\)};

	\draw[jump,bend right] (i1) to (j1);
	\draw[pairing,bend right] (j1) to (i1);
	\draw[jump,bend right] (i2) to (j2);
	\draw[pairing,bend right] (j2) to (i2);

	\node at (1,-0.7) {\(\Trace(\vb{A}_1)\Trace(\vb{A}_2)\)};
      \end{scope}
      \begin{scope}[xshift=4cm]
	\node (i1) at (0,0) {\(i_1\)};
	\node (i2) at (2,0) {\(i_2\)};
	\node (j1) at (0,2) {\(j_1\)};
	\node (j2) at (2,2) {\(j_2\)};

	\draw[jump] (i1) to (j1);
	\draw[pairing] (j1) to (i2);
	\draw[jump] (i2) to (j2);
	\draw[pairing] (j2) to (i1);

	\node at (1,-0.7) {\(\Trace(\vb{A}_1\vb{A}_2)\)};
      \end{scope}
      \begin{scope}[xshift=8cm]
	\node (i1) at (0,0) {\(i_1\)};
	\node (i2) at (2,0) {\(i_2\)};
	\node (j1) at (0,2) {\(j_1\)};
	\node (j2) at (2,2) {\(j_2\)};

	\draw[jump] (i1) to (j1);
	\draw[pairing,bend right] (j1) to (j2);
	\draw[jump] (j2) to (i2);
	\draw[pairing,bend right] (i2) to (i1);

	\node at (1,-0.7) {\(\Trace(\vb{A}_1\vb{A}_2^*)\)};
      \end{scope}
    \end{tikzpicture}
  \]
\end{exam}

\begin{exam}
  The \(15\) basic \(\GO(n)\)-invariants tensors \(\End(V)^{\otimes 3}\to\R\) are
  \begin{multline*}
    \{\Trace(\vb{A}_1)\Trace(\vb{A}_2)\Trace(\vb{A}_3);\Trace(\vb{A}_i)\Trace(\vb{A}_j\vb{A}_k);\Trace(\vb{A}_i)\Trace(\vb{A}_j\vb{A}_k^*);
      \\
    \Trace(\vb{A}_i\vb{A}_j\vb{A}_k);\Trace(\vb{A}_i^*\vb{A}_j\vb{A}_k);\Trace(\vb{A}_i\vb{A}_j^*\vb{A}_k);\Trace(\vb{A}_i\vb{A}_j\vb{A}_k^*)\}.
  \end{multline*}
  It is in bijection with the set of \(2\)-partitions of \(\Set{i_1,i_2,i_3,j_1,j_2,j_3}\).
  The tensors in the first line are obained from the tensors of type \((2,0)\) by adding a trivial loop.
  The four tensors in the second line corresponds to the loops:
  
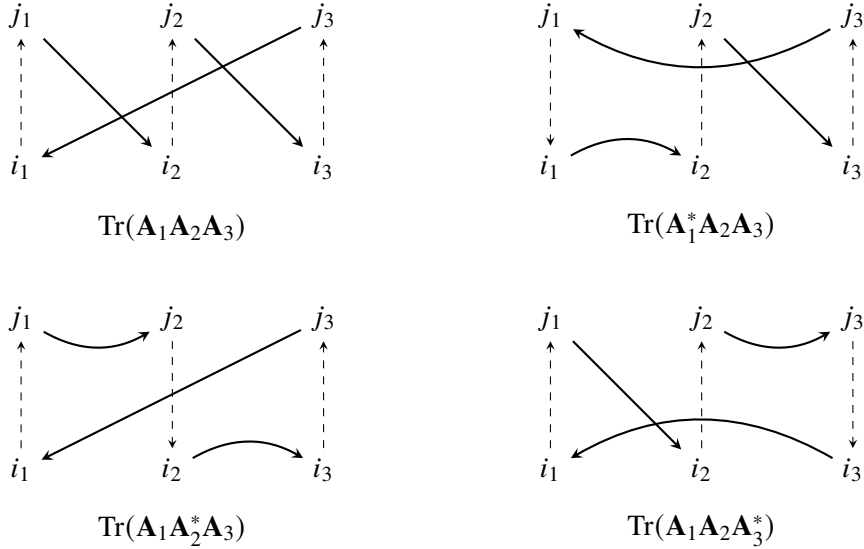
\[
    \begin{tikzpicture}
      [>=stealth,pairing/.style={thick,->}, jump/.style={dashed,->}]
      \begin{scope}
	\node (i1) at (0,0) {\(i_1\)};
	\node (i2) at (2,0) {\(i_2\)};
	\node (i3) at (4,0) {\(i_3\)};
	\node (j1) at (0,2) {\(j_1\)};
	\node (j2) at (2,2) {\(j_2\)};
	\node (j3) at (4,2) {\(j_3\)};

	\draw[jump](i1)--(j1);
	\draw[pairing] (j1) -- (i2);
	\draw[jump](i2)--(j2);
	\draw[pairing] (j2) -- (i3);
	\draw[jump](i3)--(j3);
	\draw[pairing] (j3) -- (i1);

	\node at (2,-0.8) {\(\Trace(\vb{A}_1\vb{A}_2\vb{A}_3)\)};
      \end{scope}
      \begin{scope}[xshift=7cm]
	\node (i1) at (0,0) {\(i_1\)};
	\node (i2) at (2,0) {\(i_2\)};
	\node (i3) at (4,0) {\(i_3\)};
	\node (j1) at (0,2) {\(j_1\)};
	\node (j2) at (2,2) {\(j_2\)};
	\node (j3) at (4,2) {\(j_3\)};

	\draw[jump](j1)--(i1);
	\draw[pairing,bend left] (i1) to (i2);
	\draw[jump](i2)--(j2);
	\draw[pairing] (j2) -- (i3);
	\draw[jump](i3)--(j3);
	\draw[pairing,bend left] (j3) to (j1);

	\node at (2,-0.8) {\(\Trace(\vb{A}_1^*\vb{A}_2\vb{A}_3)\)};
      \end{scope}
      \begin{scope}[yshift=-4cm]
	\node (i1) at (0,0) {\(i_1\)};
	\node (i2) at (2,0) {\(i_2\)};
	\node (i3) at (4,0) {\(i_3\)};
	\node (j1) at (0,2) {\(j_1\)};
	\node (j2) at (2,2) {\(j_2\)};
	\node (j3) at (4,2) {\(j_3\)};

	\draw[jump](i1)--(j1);
	\draw[pairing,bend right] (j1) to (j2);
	\draw[jump](j2)--(i2);
	\draw[pairing,bend left] (i2) to (i3);
	\draw[jump](i3)--(j3);
	\draw[pairing] (j3) -- (i1);

	\node at (2,-0.8) {\(\Trace(\vb{A}_1\vb{A}_2^*\vb{A}_3)\)};
      \end{scope}
      \begin{scope}[xshift=7cm,yshift=-4cm]
	\node (i1) at (0,0) {\(i_1\)};
	\node (i2) at (2,0) {\(i_2\)};
	\node (i3) at (4,0) {\(i_3\)};
	\node (j1) at (0,2) {\(j_1\)};
	\node (j2) at (2,2) {\(j_2\)};
	\node (j3) at (4,2) {\(j_3\)};

	\draw[jump](i1)--(j1);
	\draw[pairing] (j1) -- (i2);
	\draw[jump](i2)--(j2);
	\draw[pairing,bend right] (j2) to (j3);
	\draw[jump](j3)--(i3);
	\draw[pairing,bend right] (i3) to (i1);

	\node at (2,-0.8) {\(\Trace(\vb{A}_1\vb{A}_2\vb{A}_3^*)\)};
      \end{scope}
    \end{tikzpicture}
  \]
\end{exam}

\section{Isotropic invariants of second order tensors.}
In this last section, we illustrate how to recover polynomial invariants from the tensor invariants on a chosen non-trivial example.  The direct computation of bi-invariants would be difficult, so this is a good demonstration of the relevance of Weyl's method.

\begin{theo}[\cite{Boe87}*{p.~157}]
  Let \(V=\End(\R^3)\), and let \(G=\GO(3)\), acting on \(V\) by conjugation.
  The algebra \(\R[V]^G\) is generated by the seven basic invariants 
  \[
    \{\det(\vb{A}),\Trace(\vb{A}),\Trace(\vb{A}^2),\Trace(\vb{A}\vb{A}^*),\Trace(\vb{A}^2\vb{A}^{*}),\Trace(\vb{A}^2\vb{A}^{*2}),\Trace(\vb{A}\vb{A}^*\vb{A}^2\vb{A}^{*2})\}.
  \]
\end{theo}
\begin{proof}
  The basic tensor invariants of \(\End(\R^3)\) are the trace, the transpose, and the matrix product.
  Therefore \(\R[V]^G\) is generated by the trace of words made of \(\vb{A}\), and \(\vb{A}^*\):
  \[
    \Trace(\vb{A}^{\ell_1}\vb{A}^{*\ell_1'}\dotsm\vb{A}^{\ell_k}\vb{A}^{*\ell_k'}).
  \]
  Since
  \[
    \Trace(PQR)
    =
    \Trace(RPQ),
    \qand
    \Trace(PQ)
    =
    \Trace(Q^*P^*),
  \]
  one can always assume that \(\ell_1>0\), and also that if \(k\geq 2\) then \(\ell_k'>0\).
  By Cayley--Hamilton, for any matrix \(M\)
  \[
    M^3-\Trace(M)M^2+\frac{1}{2}(\Trace(M)^2-\Trace(M^2))M-\det(M)I.
  \]
  Therefore, if we add \(\det(\vb{A})\) to the family of invariants, by linearity, we can restrict to words with \(\ell_i\leq2\).

  The generating family of invariants that we get is still infinite.
  To simplify it further, one uses trace identities (\cites{RS59,P76,Boe87}).
  The deviatoric part of (\(6\times\)) the Cayley--Hamilton identity is
  \[
    \big(6M^3 -2\Trace(M^3)\big)
    -
    \Trace(M)\big(6M^2-2\Trace(M^2)\big)
    +
    (\Trace(M)^2-\Trace(M^2))\big(3M-\Trace(M)\big)
    =
    0.
  \]
  Multiplying by a matrix \(S\) and taking the trace, we get that
  \begin{multline*}
    \Trace(M)^3\Trace(S)
    -
    3\Trace(MS)\Trace(M)^2
    -
    3\Trace(M^2)\Trace(M)\Trace(S)
    +
    3\Trace(MS)\Trace(M^2)
    \\
    +
    2\Trace(M^3)\Trace(S)
    +
    6\Trace(M^2S)\Trace(M)
    -6\Trace(M^3S)
    =
    0.
  \end{multline*}
  Using our previous notation \(T^\sigma\), this can be recast as
  \[
    \sum_{\sigma\in\mathfrak{S}_4}
    (-1)^\sigma T^\sigma(M,M,M,S)
    =
    0.
  \]
  By complete polarization, for any matrices \(P,Q,R,S\), one has
  \[
    \sum_{\sigma\in\mathfrak{S}_4}
    (-1)^\sigma T^\sigma(P,Q,R,S)
    =
    0.
  \]
  This can be written explicitly as the (24 terms) relation:
  \begin{multline*}
    \Trace(P)\Trace(Q)\Trace(R)\Trace(S)
    - \Trace(PS)\Trace(Q)\Trace(R) - \Trace(QS)\Trace(P)\Trace(R) - \Trace(RS)\Trace(P)\Trace(Q)
    \\
    - \Trace(PQ)\Trace(R)\Trace(S) - \Trace(PR)\Trace(Q)\Trace(S) - \Trace(RQ)\Trace(P)\Trace(S)
    \\
    + \Trace(PS)\Trace(RQ) + \Trace(QS)\Trace(PR) + \Trace(RS)\Trace(PQ)
    \\
    +
    \Trace(RPQ)\Trace(S) + \Trace(PRQ)\Trace(S) + \Trace(PQS)\Trace(R) + \Trace(RSQ)\Trace(P) 
    \\
    + \Trace(RPS)\Trace(Q) + \Trace(PSQ)\Trace(R) + \Trace(PRS)\Trace(Q) + \Trace(RQS)\Trace(P)
    \\
    - \Trace(PQRS) - \Trace(RPSQ) - \Trace(RPQS) - \Trace(PRSQ) - \Trace(PRQS) - \Trace(RQPS)
    =
    0,
  \end{multline*}
  which is valid for any matrices \(P,Q,R,S\).
  We have ordered the letters in each monomials in an order that will soon be relevant.

  Let us use this relation to prove that there is a family of generators with \(k\geq 2\).
  We assume that \(k\geq 2\), and for any \(2\leq i\leq k\), we take
  \[
    P=\vb{A}^{\ell_1};
    Q=\vb{A}^{*\ell_1'}\dotsm\vb{A}^{*\ell_{i-1}};
    R=\vb{A}^{\ell_i};
    S=\vb{A}^{*\ell_i'}\dotsm\vb{A}^{\ell_k}\vb{A}^{*\ell_k'}.
  \]
  In the above identity, the only words with as many blocks as \(PQRS\) are \(PQRS\) and \(RQPS\). All the other can be contracted (the only non-obvious case is maybe \(\Tr(QS)\), for which one has maybe to use the relations \(\Tr(MN)=\Tr(NM)\) or \(\Trace(M^*)=\Tr(M)\)).
  If \(P=R\), then \(PQRS=RQPS\) and we get that \(\Tr(PQRS)\) can be written with traces of words having less than  \(k\) blocks. 
  By induction, we can use words where \(\ell_i\neq\ell_j\), for all \(i\neq j\). Since \(1\leq \ell_i,\ell_j\leq 2\), we get that we can use with at most two blocks. As a byproduct of the proof, we can also chose that \(\ell_1=1\) and \(\ell_2=2\) when \(k=2\), since exchanging \(P\) and \(R\) is always possible up to shorter words.

  By transposition, the same applies to the exponents of \(\vb{A}^*\): these must be differents and one can exchange them up to traces of shorter words. Hence we also chose \(\ell_1'=1,\ell_2'=2\), if \(k=2\).

  The remaining generators in our family are
  \[
    \det(\vb{A}),
    \Trace(\vb{A}),
    \Trace(\vb{A}^2),
    \Trace(\vb{A}\vb{A}^{*}),
    \Trace(\vb{A}\vb{A}^{*2})= \Trace(\vb{A}^2\vb{A}^{*}),
    \Trace(\vb{A}^2\vb{A}^{*2}),
    \Trace(\vb{A}\vb{A}^{*}\vb{A}^2\vb{A}^{*2}).
  \]
  This finishes the proof.
\end{proof}

\begin{exam}
  For \(P=R=\vb{A}\) and \(Q=S=\vb{A}^*\), we get the non-obvious algebraic relation
  \begin{multline*}
    \Trace(\vb{A}\vb{A}^*\vb{A}\vb{A}^*)
    =
    \frac{1}{2}
    \big(
      \Trace(\vb{A})^2
      -
      2\Trace(\vb{A}^2)\Trace(\vb{A})^2
      -
      4\Trace(\vb{A}\vb{A}^*)\Trace(\vb{A})^2
      \\
      +
      \Trace(\vb{A}^2)^2
      +
      2
      \Trace(\vb{A}\vb{A}^*)^2
      +
      8\Trace(\vb{A}^2\vb{A}^*)\Trace(\vb{A})
      -
      4
      \Trace(\vb{A}^2\vb{A}^{*2})
    \big).
  \end{multline*}
\end{exam}

\begin{rema}
  See~\cite{Boe87} for joint invariants of several matrices, in which Spencer proves the following interesting more general lemma: In \(\R^3\), the trace of any matrix product of degree seven (or more) can be expressed as a polynomial in traces of matrix products of degree less than seven. 
\end{rema}

\appendix
\part*{Appendix.}
\section{A constructive proof of the complete reducibility of \texorpdfstring{\(\GL(m)\)}{GL(m)}-representations.}
\label{sec:reducibility}
Recall that we denote by \(\Gamma\bydef\GL(m)\), the group of invertible real \((m\times m)\)-matrices.
In this appendix, we prove the following fundamental result.
\begin{theo}[Complete reducibility]\label{theo:complete_reducibility}
  Every finite-dimensional rational \(\Gamma\)-representation decomposes as a direct sum of irreducible rational \(\Gamma\)-representations.
\end{theo}
To our knowledge, the proof we provide here is new. It is somewhat ``minimal'', and just uses elementary material. There are several proofs commonly found in the literature. We (very roughly) mention two of these.

\begin{itemize}
  \item using Weyl's ``unitary trick'' over $\mathbb C$, and (a hint of) Galois theory  to descend from $\mathbb C$ to $\R$.
  \item via Schur--Weyl duality.
\end{itemize}
In our approach, all these ingredients are removed.\\
We first reformulate the theorem. Let \(V\) be a \(\Gamma\)-representation, and let \(W\) be a subrepresentation of \(V\).

From linear algebra, recall that a linear map \(\pi\colon V \to W\) is called a \textsl{projector onto \(W\)} if \(\pi\vert_{W}=\id_W\).
In that case, the subspace \(\ker(\pi)\) is supplementary to \(W\).
Conversely, if \(V=W\oplus W'\), then the associated projection \(V\to W\) is a projector onto \(W\).
To sum up, direct sum decompositions of \(V\) are in one-to-one correspondance with projectors in \(\End(V)\).

Recall that if \(V\) is a \(\Gamma\)-representation, then \(\End(V)\) is naturally a \(\Gamma\)-representation, for the conjugation action. More generally, let \(\rho_V\colon\Gamma\to\GL(V)\) and \(\rho_W\colon\gamma\to\GL(W)\) be two \(\Gamma\)-representations. Then the vector space \(\Hom(V,W)\) of linear maps from \(V\) to \(W\) is naturally a representation, for the action defined by
\[
  \gamma\ast u
  \bydef
  \rho_W(\gamma)\circ u\circ\rho_V(\gamma)^{-1}.
\]
An invariant linear map \(u\in\Hom(V,W)\) for this action is usually called \textsl{\(\Gamma\)-equivariant}.
A linear map \(u\in\Hom(V,W)\) is thus equivariant if it commutes with the \(\Gamma\)-actions:
\[
  \forall\gamma\in\Gamma\qcomma
  \forall \vb*{v}\in V\qcomma
  u(\gamma\ast\vb*{v})
  =
  \gamma\ast u(\vb*{v}).
\]
We already have seen a lot of examples of such maps throughout the text.

If a projector \(\pi\in\End(V)\) onto a subrepresentation \(W\) is equivariant, then its kernel \(W'\bydef\ker(\pi)\) is \(\Gamma\)-stable, since
\[
  \forall \vb*{v}\in\ker\pi\qcomma
  \forall\gamma\in\Gamma\qcomma
  \pi(\gamma\ast\vb*{v})
  =
  \gamma\ast\pi(\vb*{v})
  =
  0.
\]
One has thus \(V=W\stackrel\Gamma\oplus W'\), as \(\Gamma\)-representations.
Conversely, if \(V=W\stackrel\Gamma\oplus W'\), then the natural projection on \(W\) is equivariant.
To sum up, direct sum decompositions of \(V\) into subrepresentations are in one-to-one correspondance with equivariant projectors in \(\End(V)\).

By induction on the dimension, it is clear that if every subrepresentation admits a supplementary subrepresentation, then every finite-dimensional representation decomposes as a direct sum of irreducible subrepresentations. The reverse implication is obvious.
We will hence prove the following equivalent version of Theorem~\ref{theo:complete_reducibility}.
\begin{theo}\label{theo:rep_supp}
  Every subrepresentation \(W\) of a finite-dimensional rational \(\Gamma\)-representation \(V\) admits a supplementary subrepresentation.
\end{theo}

To motivate our next step, a short digression on the classical notion of Reynolds operators may be helpful.
Among the many contexts in which these operators exist, consider the case of a compact real matrix group \(G\), and a  finite-dimensional \(G\)-representation \(V\). Then the so-called \textit{Reynolds operator} is the averaging operator $\rey_V$, along \(G\)-orbits:
\[
  \rey_V\colon V \to V^G
  \qcomma
  \vb*{v} \mapsto \int_{g \in G} (g \ast \vb*{v})d \mu,
\]
where \(\mu\) is the normalised Haar mesure on \(G\) (i.e. with \(\mu(G)=1\)).
In case \(G\) is a finite  group with  \(N\) elements, the formula is simply 
\[
  \rey_V(\vb*{v})
  =
  \frac 1 N \sum_{g \in G} g \ast  \vb*{v}.
\] 
It is classical, that \(\rey_V\) is the unique \(G\)-equivariant projector onto the \(G\)-invariant subspace \(V^G\subset V\).
This analogy seems very limited: indeed, in our context, the group \(\Gamma=\GL(m)\) is not compact. What seems worse, it  has no ``large enough'' compact subgroup that would allow to use Weyl's unitary trick.
As a consequence, the previous averaging formula does not make sense at all.
Perhaps surprisingly,  Reynolds operators still exist---but in a purely algebraic sense. \\
From now on, for a (possibly infinite-dimensional) \(\Gamma\)-representation \(V\), let us call \textit{Reynolds operator} a \(\Gamma\)-equivariant projector  onto the subrepresentation \(V^\Gamma\), 
\[
  \rey_V
  \colon
  V \to V^\Gamma.
\]
By the above considerations, the theorem that we want to prove implies the existence of Reynolds operators as a particular case.
In order to get a reciprocal, we need to strengthen the notion by assuming the compatibility of Reynolds operators with equivariant maps, as follows.\\
Let \(\mathcal{F}\) be a family of \(\Gamma\)-representations such that for any representation \(V\in\mathcal{F}\),
all subrepresentations \(W\) of \(V\) are in \(\mathcal{F}\)
and all quotient representations \(W\) of \(V\) (that is the images of surjective equivariant maps \(V\to W\)) are in \(\mathcal{F}\).
Such a family is said to be \textit{closed under subrepresentations and quotients}.
A \textsl{Reynolds operator} \(\rey\) \textsl{on} \(\mathcal{F}\) is 
the assignement of a Reynolds operator \(\rey_V\colon V\to V^\Gamma\) to each \(\Gamma\)-representation \(V\in\mathcal{F}\), in such way that for any \(\Gamma\)-equivariant linear map \(\mu\colon V\to W\), one has
\[
  \mu\circ\rey_{V}
  =
  \rey_{\mu(W)}\circ\mu.
\]
If \(W\) itself is in \(\mathcal{F}\), this can be simplified as  \(\mu\circ\rey_{V} = \rey_{W}\circ\mu\).
\begin{exam}\label{examFr}
  For $r \geq 0,$ let \(\mathcal{F}_r\) be the family of finite-dimensional rational \(\Gamma\)-representations \(\rho\colon \Gamma\to\GL(V)\), such that \(\det(\gamma)^r\rho(\gamma)\) is polynomial. This family is closed under subrepresentations and quotients. We make  three elementary albeit important observations. Their verification is left to the reader. \begin{enumerate}
    \item The family \(\mathcal{F}_0\) consists of all polynomial representations.
    \item For $r \leq s$, we have   \(\mathcal{F}_r \subset \mathcal{F}_s \).
    \item For every (rational) representation $V$ of $\Gamma$, there exists an $r \geq 1$ such that \( V \in \mathcal{F}_r.\)
  \end{enumerate}  

\end{exam}

It is classical that the existence of a Reynolds operator implies complete reducibility.
\begin{theo}\label{theoReyCR}
  Let $r\geq 1$ be an integer. Assume there is a Reynolds operator on \(\mathcal{F}_r\). Let \(V\) be a rational \(\Gamma\)-representation , such that  \(\End(V)\in\mathcal{F}_r\). Then  every subrepresentation \(W \subset V\) admits a supplementary subrepresentation.
\end{theo}
\begin{proof}
  Let \(V\) be a finite-dimensional rational representation of \(\Gamma\), and let \(W\subseteq V\) be a subrepresentation.
  Recall that \(\End(V)\) is a rational \(\Gamma\)-representation, for the conjugation action.
  Since \(W\) is a \(\Gamma\)-stable subspace, we get that \(\Hom(V,W)\subseteq\End(V)\) is a (finite-dimensional rational) subrepresentation. Also, the restriction map \(\Hom(V,W)\to\End(W)\) is surjective and \(\Gamma\)-equivariant. As such, it presents $\End(W)$ as a quotient of $\Hom(V,W)$.
  Assume \(\End(V)\in\mathcal{F}_r\). By the above,  it follows that \(\Hom(V,W)\) and \(\End(W)\) are also in \(\mathcal{F}_r\).
  Let \(\pi\in\Hom(V,W)\) be any projector onto \(W\).
  Then \(\rey_{\Hom(V,W)}(\pi)\) is an equivariant linear map \(V\to W\).
  The restriction map \(\Hom(V,W)\to\End(W)\) is \(\Gamma\)-equivariant,
  and the identity map is invariant by conjugation,
  therefore
  \[
    (\rey_{\Hom(V,W)}(\pi))\vert_W
    =
    \rey_{\End(W)}(\id_W)
    =
    \id_W,
  \]
  by the compatibility of Reynolds operators with \(\Gamma\)-equivariant linear maps.
  We have produced an equivariant projector \(V\to W\).
\end{proof}

\begin{rema}
  In the premises of Theorem \ref{theoReyCR}, the assumption  \(\End(V)\in\mathcal{F}_r\) is essential, and cannot be replaced by the (inappropriate) assumption  \(V \in\mathcal{F}_r\). Observe that none of these conditions imply the other. For instance, the standard representation $V=\R^m$ is polynomial (=belongs to  $\mathcal{F}_0$), but  \(\End(V) = \operatorname{M}_{m}\) (with the conjugation action) does not: it belongs to $ \mathcal{F}_1\setminus\mathcal{F}_0$. On the other hand,  for $r \geq 1$, the homomorphism $\det^{-r}\colon\GL(m) \to \R^*$  gives a  one-dimensional rational representation $V$,  that belongs to $ \mathcal{F}_r\setminus\mathcal{F}_{r-1}$. However, in this case $\End(V)$ is the trivial one-dimensional representation, that is obviously polynomial.
\end{rema}
\begin{rema}\label{unirey}
  There is a converse to the Theorem above: provided \(\Gamma\) is completely reducible, then all Reynolds operators \(\rey_V\) exist and are unique. Proving this  is a  good (but not so easy) exercise, left to  the reader. The same equivalence is true for any matrix group.
\end{rema}

\begin{rema}
  Some  matrix groups  are  not completely reducible; typically the upper unitriangular group \(U_2\simeq(\R,+)\).
\end{rema}

Recall that a rational representation is a group homomorphism \(\rho\colon \Gamma\to\GL(V)\), such that \(\rho(\gamma)\) is rational in \(\gamma\). 
Recall also the notation \(\R[\operatorname{M}_{m}]=\R[x_{ij}, 1 \leq i,j \leq n]\) for the algebra of polynomial functions on \(\Gamma\), and \(\O(\Gamma)\) for the algebra of regular functions.
\begin{exam}
  The function \[\vb M \mapsto \frac {x_{11}^2+ 2x_{m1}^2 x_{mm}-3} {\det(M)^3}  \in \R\] is regular, but it is not polynomial.
\end{exam}

If \((\vb*{e}_1,\dotsc,\vb*{e}_n)\) is a basis of \(V\), one has for all \(\vb*{v}\in V\)
\[
  \gamma\ast\vb*{v}
  =
  \sum_{i=1}^n
  \alpha^i(\vb*{v})(\gamma)\vb*{e}_i,
\]
for some rational functions \(\gamma\mapsto\alpha^i(\vb*{v})(\gamma)\).
Note that \(\vb*{v}\) is invariant if and only if the rational functions \(\alpha^i(\vb*{v})\in\O(\Gamma)\) are all constant.
To a Reynolds operator \(\rey_V\) and a vector \(\vb*{v}\), such that
\[
  \rey_V(\vb*{v})
  =
  \sum_{i=1}^n
  a^i(\vb*{v})\vb*{e}_i,
\]
one can hence associate \(n\) ``averaging'' maps
\[
  \O(\Gamma)\to\R\qcomma
  \alpha^i(\vb*{v}) \mapsto a^i(\vb*{v}),
\]
such that \(a^i\) is an invariant of \(V\).
Our goal is to show that there is a universal reverse procedure producing a Reynolds operator.

%The action of \(\Gamma\) on itself by right translations \(\gamma\ast\gamma'=\gamma'\gamma^{-1}\), induces a linear action on  \(\O(\Gamma)\), by means of the formula,  for \(f \in \O(\Gamma)\):
There is a natural action of \(\Gamma\) on  \(\O(\Gamma)\), by means of the formula,  for \(f \in \O(\Gamma)\):
\[
  \forall\gamma,\gamma'\in\Gamma\qcomma
  (\gamma\ast f )(\gamma')=f(\gamma'\gamma).
\]
In that way,  \(\O(\Gamma)\)  becomes a (rational, infinite-dimensional) representation of \(\Gamma\). 
By transitivity of the action of \(\Gamma\) on itself, it is straightforward that the fixed points in \(\O(\Gamma)\) are the constant maps.
Therefore a Reynolds operator \(\rey_{\O(\Gamma)}\) for \(\O(\Gamma)\) would be the perfect candidate for an ``averaging'' operator.

We have already seen that the denominator of a regular map \(\Gamma\to\R\) is a power of the polynomial 
\( D(\gamma) \bydef \det(\gamma)\):
\[
  \R[\operatorname{M}_{m}]\subseteq\O(\Gamma)=\R[\operatorname{M}_{m}][D^{-1}]\subseteq\R(\operatorname{M}_{m}).
\]
Hence there is an infinite increasing sequence of subrepresentations
\[
  \R[\operatorname{M}_{m}]
  \subseteq
  D^{-1}\R[\operatorname{M}_{m}]
  \subseteq\dotsb\subseteq
  D^{-r}\R[\operatorname{M}_{m}]
  \subseteq\dotsb\subseteq
  \O(\Gamma).
\]
It is exhaustive, meaning that \( \bigcup_{r \geq 0} D^{-r}\R[\operatorname{M}_{m}]=\O(\Gamma)\).
For the sake of simplicity, we will restrict ourselves to construct a Reynolds operator on \(D^{-r}\R[\operatorname{M}_{m}]\).
We will come back to this matter later.

\begin{theo}\label{theo:rey_Fr}
  If there is a Reynolds operator on the \(\Gamma\)-representation \(D^{-r}\R[\operatorname{M}_{m}]\),
  then there is a Reynolds operator on the family \(\mathcal{F}_r\). 
\end{theo}
\begin{proof}
  Let \(V\in\mathcal{F}_r\) be a finite-dimensional rational representation of \(\Gamma\).
  Let \((\vb*{e}_1,\dotsc,\vb*{e}_n)\) be a basis of \(V\).
  For \(\vb*{v}\in V\) and for any \(j=1,\dotsc,n\), there is a rational map
  \[
    \alpha^j(\vb*{v})\colon\gamma \mapsto \vb*{e}^j(\gamma\ast\vb*{v}).
  \]
  If \(\vb*{v}\) is invariant, this map is constant. More generally, one has
  \[
    \gamma\ast\alpha^j(\vb*{v})(\gamma')
    =
    \alpha^j(\vb*{v})(\gamma^{-1}\ast\gamma')
    =
    \alpha^j(\vb*{v})(\gamma'\gamma)
    =
    \alpha^j(\gamma\ast\vb*v)(\gamma').
  \]
  In other words the linear map
  \[
    \alpha^j
    \colon
    V\to\O(\Gamma)
    \qcomma
    \vb*{v}\mapsto \alpha^j(\vb*{v})
  \]
  is \(\Gamma\)-equivariant.
  Since
  \[
    \vb*{v}=\sum_j\alpha^j(\vb*{v})(\id)\vb*{e}_j,
  \]
  we get an injective \(\Gamma\)-equivariant linear map %depending on the basis \(\vb*{e}\)
  \[
    \alpha_{\vb*{e}}\colon V\mapsto (D^{-k}\R[\operatorname{M}_{m}])^n\qcomma
    \vb*{v}\mapsto(\alpha^1(\vb*{v}),\dotsc,\alpha^n(\vb*{v})).
  \]

  Applying the Reynolds operator \(\rey\colon D^{-k}\R[\operatorname{M}_{m}]\to\R\) on the right hand side, we obtain constants 
  \[
    a^j(\vb*{v})
    \bydef
    \rey(\alpha^j(\vb*{v}))\in\R.
  \]
  We define an endomorphism \(\rey_V\in\End(V)\) by setting
  \[
    \rey_V(\vb*{v})
    =
    \sum_{j=1}^n a^j(\vb*{v})\vb*{e}_j.
  \]
  For every \(\vb*{v}\in V\), the vector \(\rey_V(\vb*{v})\) is invariant, because the preimage of an invariant by an equivariant injective map is invariant.
  We have already noticed that if \(\vb*{v}\) is invariant, then \(\alpha^j(\vb*{v})=a^j(\vb*{v})\) is constant, which shows that \(\rey_V\) is a projector on \(V^\Gamma\).
  Furthermore, by equivariance of the Reynolds operator \(\rey\),
  \[
    a^j(\gamma\ast\vb*{v})
    =
    \rey(\alpha^j(\gamma\ast\vb*{v}))
    =
    \rey(\gamma\ast\alpha^j(\vb*{v}))
    =
    \rey(\alpha^j(\vb*{v}))
    =
    a^j(\vb*{v}).
  \]
  Therefore \(\rey_V\) is equivariant. To sum up, \(\rey_V\) is a Reynolds operator on \(V\). If \((\bar{\vb*{e}}_1,\dotsc,\bar{\vb*{e}}_n)\) is another basis of \(V\), then
  \[
    \bar{\vb*{e}}_j
    =
    \sum_i
    p_j^i
    \vb{e}_i
    \qand
    \bar{\vb*{e}}^j
    =
    \sum_k
    q_k^j
    \vb{e}^k,
  \]
  for two matrices \(P,Q\) such that \(P=Q^{-1}\).
  then
  \[
    \bar\rey_V(\vb*{v})
    =
    \sum_{j}
    \bar{a}^j(\vb*{v})
    \bar{\vb*{e}}_j
    =
    \sum_{i,j,k}
    q_k^j
    p_j^i
    a^k(\vb*{v})
    \vb*{e}_i
    =
    \sum_{i}
    a^i(\vb*{v})
    \vb*{e}_i
    =
    \rey_V(\vb*{v}).
  \]
  Therefore the projector \(\rey_V\) does not depend on the choice of a basis.

  Now let \(\mu\colon V\to W\) be a \(\Gamma\)-equivariant linear map between two finite-dimensional rational representations.
  We work with two basis \(\vb*{e},\vb*{f}\) of \(V\) and \(W\), in which \(\mu\) has components \(m_i^j\).
  The equivariance of \(\mu\) is recast in components as
  \[
    \forall i,k\qcomma
    \sum_j m_j^k\alpha^j(\vb*{e}_i)
    =
    \sum_j \alpha^k(m_i^j\vb*{f}_j).
  \]
  One infers the same relationship for the coefficients \(a^j\), and the relation \(\mu\circ\rey_V=\rey_W\circ\mu\) follows.
\end{proof}

\begin{rema}
  In the previous proof, we picked a basis \((\vb*{e}_1,\dotsc,\vb*{e}_n)\), for the sake of providing concrete formulas. It is also possible to provide a more abstract, ``coordinate-free'' proof, making it transparent that \(\rey_V\) is independent of the choice of a basis.

\end{rema}

The polynomial degree on \(\R[\operatorname{M}_{m}]\) induces a polynomial degree on \(\R(\operatorname{M}_{m})\) and on \(\O(\Gamma)\).
For a polynomial \(P\) of degree \(d\), one has
\[
  \deg(D^{-r}P)
  =
  d-rm.
\]
The algebra \(\O(\Gamma)\) is graded by its homogeneous components. Here the degree is allowed to be negative:
\[
  \O(\Gamma) 
  =
  \bigoplus_{d\in \mathbb Z} \O(\Gamma)_d.
\] 
There is the following useful characterisation of homogeneous elements: \(f \in \O(\Gamma)\) is homogeneous of degree \(d\), if and only if \(f(\lambda \gamma)=\lambda^d f(\gamma)\), for all \(\lambda \in \R\) and all \(\gamma \in \Gamma\).

\begin{exam}
  The determinant  is a polynomial map, homogeneous of degree \(m\). As such, denote it by \[D=D(x_{ij}) \in \R[x_{ij}]_m.\]
\end{exam}
\begin{exam}
  When \(m=2\),  the homogeneous fraction \[ \frac {x_{13}^2x_{21}-x_{22}^3 } { x_{11} x_{22}-  x_{12} x_{21}}\]  belongs to \( \O(\Gamma)_1\), but it is not a polynomial.
  This illustrates that for \(m\geq2\), one has \[\R[\operatorname{M}_{m}]_d \subsetneq \O(\Gamma)_d,\] for all \(d\).
\end{exam}

By definition of \(\O(\Gamma)\) and of the degree, one gets another exhaustive sequence of homogeneous \(\Gamma\)-stable subspaces

\[
  \R[\operatorname{M}_{m}]_d
  \subseteq
  D^{-1}\R[\operatorname{M}_{m}]_{d+m}
  \subseteq\dotsb\subseteq
  D^{-r}\R[\operatorname{M}_{m}]_{d+rm}
  \subseteq\dotsb\subseteq
  \O(\Gamma)_d.
\]
Its $r$-th term is just the subspace of homogeneous elements of degree $d$, in $ D^{-r}\R[\operatorname{M}_{m}]$.
That these subspaces are indeed  \(\Gamma\)-stable, follows from  the characterisation of homogeneous elements given above.

The next observation comes in handy. Since invariant rational functions are constant, and since the degree of a constant is $d=0$, there is no non-zero fixed point in the representation \(\O(\Gamma)_d\) (and of its subrepresentations) when \(d\neq0\).

We can now  provide a short and constructive proof of the following result.
\begin{theo} \label{theo:rey_Dr}
  There exist Reynolds operators \(\rey_{(r)}\) on the \(\Gamma\)-representation \(D^{-r}\R[\operatorname{M}_{m}]\), for \(r\geq0\).
\end{theo}

\begin{proof}
  The homogeneous equivariant decomposition yields a direct sum decomposition
  \[
    D^{-r}\R[\operatorname{M}_{m}] =
    \bigoplus_{d\in\mathbb{Z}}
    D^{-r}\R[\operatorname{M}_{m}]_d
    =
    \bigoplus_{d\in\mathbb{Z}}
    D^{-r}\R[\operatorname{M}_{m}]_{d+rm}.
  \]
  Observe that \[  D^{-r}\R[\operatorname{M}_{m}]_{d+rm}= \left( D^{-r}\R[\operatorname{M}_{m}]  \right) \cap \O(\Gamma)_d. \]
  Since the invariants of $\O(\Gamma)$ are the constant functions, homogeneous of degree $d=0$, it follows that  $\O(\Gamma)_d$  (and hence $ D^{-r}\R[\operatorname{M}_{m}]_{d+rm}$) has non non-zero $\Gamma$-invariant element if $d \neq 0$. 
  Hence, we first reduce the problem to the finite-dimensional representation \(D^{-r}\R[\operatorname{M}_{m}]_{rm}\).

  For any \(f\in D^{-r}\R[\operatorname{M}_{m}]\), there is a homogeneous equivariant decomposition \(f=\sum_{d\in\mathbb{Z}} f_d\), such that \(f_0\in D^{-r}\R[\operatorname{M}_{m}]_{rm}\).
  If there is a Reynolds operator \(\rey_{(r)}\) on \(D^{-r}\R[\operatorname{M}_{m}]_{rm}\), we can thus extend it to \(D^{-r}\R[\operatorname{M}_{m}]\) by setting 
  \[
    \rey_{(r)}(f)
    :=
    \rey_{(r)}(f_0).
  \]
  If \(f\) is constant, then \(f=f_0=\rey_{(r)}(f_0)=\rey_{(r)}(f)\).
  We have shown that \(\rey_{(r)}\) is a projector onto \(\R\).
  Lastly, since \(\rey_{(r)}\) can be expressed as the composition of two equivariant linear maps, it is equivariant.

  It remains to exhibit a Reynolds operator on the representation \(D^{-r}\R[\operatorname{M}_{m}]_{rm}\). This is the same as an equivariant projector 
  \[
    \rey_{(r)}\colon\R[\operatorname{M}_{m}]_{rm}\to\R\cdot D^r.
  \]

  To begin with, let us treat the case \(r=1\).

  Write $V=\R^m$.  Recall that  \(\operatorname{M}_{m}\) is isomorphic  to \(V\otimes V^\ast\): there is a natural isomorphism \[\Psi: V\otimes V^\ast \stackrel  \sim \to \End(V)=\operatorname{M}_{m},\] sending a pure tensor  \( \vb*{v} \otimes \phi \) to the endomorphism (of rank $\leq 1$) 
  \[ 
    (x \mapsto \phi(x)\vb*{v}).
  \] 
  In this proof, we consider  $V \otimes V^*$ as a representation of $\Gamma$, by means of the natural action on 
  $V$, and the trivial action on $V^*$: 
  \[ 
    \sigma \ast (\vb*{v} \otimes \phi)=(\sigma \ast \vb*{v}) \otimes \phi.
  \]
  We make $\Gamma$ act on $\operatorname{M}_{m}$ by left multiplication.
  With respect to these actions, the isomorphism $\Psi$ is $\Gamma$-equivariant, as shown by the computation 
  \[ 
    \Psi(\gamma \ast (\vb*{v} \otimes \phi))(x)
    =
    \Psi((\gamma  \ast \vb*{v}) \otimes \phi)(x)
    =
    \phi(x) \gamma \ast \vb*{v}
    =
    (\gamma \ast  \Psi(\vb*{v} \otimes \phi))(x).
  \]
  Observe that the $\Gamma$-action here, on $\End(V)$, is \textit{not} the conjugation action introduced at the beginning of this Appendix.

  From Theorem \ref{theo:triinv}, recall the  linear map (on tensors of order \(2m\)) 
  \[
    D_m\colon   V^{\otimes m} \otimes V^{*\otimes m} \to \R
  \]
  defined by the formula
  \[
    D_m(\vb*{v}_1 \otimes \dotsc  \otimes\vb*{v}_m\otimes f^1 \otimes \dotsc  \otimes f^m)
    =
    \begin{vmatrix}
      f^1(\vb*{v}_1)&\dotso&f^1(\vb*{v}_m)\\
      \\
      f^m(\vb*{v}_1)&\dotso&f^k(\vb*{v}_m)
    \end{vmatrix}.
  \]
  From the multiplicativity of the determinant, we get, for $\gamma \in \Gamma$: 
  \[  
    D_m((\gamma  \ast \vb*{v}_1) \otimes \dotsc  \otimes (\gamma \ast \vb*{v}_m) \otimes f^1 \otimes \dotsc  \otimes f^m )
    =
    \det(\gamma)  D_m(\vb*{v}_1 \otimes \dotsc  \otimes\vb*{v}_m \otimes f^1 \otimes \dotsc  \otimes f^m  )
  .\]
  This formula expresses that  $D_m$ is $\Gamma$-equivariant, with respect to the actions given\begin{itemize}
    \item at the source, by  \( \gamma \ast ( \vb*{v}_1 \otimes \dotsc  \otimes\vb*{v}_m\otimes f^1 \otimes \dotsc  \otimes f^m  )=(\gamma  \ast \vb*{v}_1) \otimes \dotsc  \otimes (\gamma \ast \vb*{v}_m)\otimes f^1 \otimes \dotsc  \otimes f^m  \)
    \item at the target $\R$, by $\rho=\det$.
  \end{itemize} 
  At the source of $D_m$, let us match the $i$-th occurence of $V^*$ with the $i$-th occurence of $V$. This amounts to pairing $f_i$ with $ \vb*{v}_i$. Via the isomorphism $\Psi$ above (used $m$ times), $D_m$ then reads as a $\Gamma$-invariant linear map 
  \[
    \tilde D_m: \operatorname{M}_{m}^{\otimes m}  \to \R
  \]
  Seen as such, it is a \textit{symmetric} multilinear $m$-form on $\End(V)=\operatorname{M}_{m}$. Indeed, one readily checks that \textit{simultaneously} permuting $\vb*{v}_i,\vb*{v}_j$ and $f_i,f_j$, does not change the value of $ D_m(\vb*{v}_1 \otimes \dotsc  \otimes\vb*{v}_m \otimes f^1 \otimes \dotsc  \otimes f^m )$. (The determinant is invariant under conjugation  by a permutation matrix.)
  Therefore, $ \tilde D_m$ gives rise to a $\Gamma$-equivariant linear map
  \[ 
    \tilde \rey_{(1)}
    \colon
    \R[\operatorname{M}_m]_{m}
    \simeq
    \Sym^m(\operatorname{M}_{m})
    \to 
    \R \cdot D.
  \]
  Fix a basis $\vb*{e}_1,\dotsc,\vb*{e}_m$ of $V$, and denote by $\vb*{e}^1,\dotsc,\vb*{e}_m$ its dual basis.  Let us  compute: 
  \begin{align*}
    \tilde \rey_{(1)}(D)
  & =
  \tilde D_m \left(\sum_{\sigma \in S_m} \epsilon(\sigma) (\vb*{e}_1 \otimes \vb*{e}^{\sigma(1)}) \otimes \dotsb \otimes (\vb*{e}_m \otimes \vb*{e}^{\sigma(m)}) \right)
\\&=
\sum_{\sigma \in S_m} \epsilon(\sigma) D_m\left(\vb*{e}_{1}  \otimes \dotsb \otimes \vb*{e}_{m} \otimes \vb*{e}^{\sigma(1)} \otimes \dotsb \otimes \vb*{e}^{\sigma(m)}\right)
\\&=
\sum_{\sigma \in S_m} \epsilon(\sigma)\det(\Mat(\delta_j^{\sigma(i)})) 
=\sum_{\sigma \in S_m} \epsilon(\sigma)^2
=
\sum_{\sigma \in S_m} 1=m!.
  \end{align*}
  As a result, 
  \[
    \rey_{(1)}\bydef\frac 1 {m!} \tilde \rey_{(1)}
  \]
  is the sought-for equivariant projector, for the subspace $\R.  D \subset \R[\operatorname{M}_m]_m. $

  The general case, for $r\geq 2$, is similar. Consider the space $W\bydef V^r$, as a representation of $\Gamma$ through the natural diagonal action: 
  \[
    g \ast (\vb*v_1, \dotsc,\vb*v_r)=(g \ast \vb*v_1, \dotsc, g \ast \vb*v_r).
  \]
  This is exactly the action described in section \ref{secinv} (the space $V^m$  thereof, being here   $V^r$).  Accordingly,  the  homomorphism giving this action,\[\rho: \Gamma=\GL(m) \to \Gamma_r:=\GL(mr),\]  is the diagonal $r$-fold matrix embedding 
  \[
    g \mapsto (g,g,\dotsc, g)
  .\]
  Applying the preceding construction to $W$ in place of $V$, yields  a natural embedding  
  \[
    \R D_W \subset  \R[\operatorname{M}_{rm}]_{rm} ,
  \] 
  where $D_W$ is the determinant of square matrices of size $mr$, and a projector
  \[
    \rey_W \left(= \frac { \tilde \rey_W} {(mr)!} \right)
    \colon
\R[\operatorname{M}_{rm}]_{rm}
    \to \R D_W
  .\]
  These  are $\Gamma_r$-equivariant. A fortiori, they are $\Gamma$-equivariant, where  $\Gamma$ acts via $\rho$.
  Consider the (linear) diagonal $r$-fold  matrix embedding 
  \[
    \operatorname{M}_{m}\to\operatorname{M}_{mr}.
  \]
  On homogeneous polynomials of degree $mr$, it induces a $\Gamma$-equivariant inclusion 
  \[
\R[\operatorname{M}_{m}]_{rm}
    =
    \Sym^{rm}(\operatorname{M}_{m} )  
    \hookrightarrow 
\R[\operatorname{M}_{rm}]_{rm}
    =
    \Sym^ {rm}(\operatorname{M}_{rm}).
  \] 
  Then, the composite inclusion 
  \[
    \R.  D^r 
    \subset
\R[\operatorname{M}_{m}]_{rm}
    \subset
\R[\operatorname{M}_{rm}]_{rm}
    \]
    is none other than  
    \[
      \R.  D_W 
      \subset
\R[\operatorname{M}_{rm}]_{rm}.
      \]
      This fact is a direct consequence of the block formula 
      \[
      \det(g,g,\dotsc, g)=\det(g)^r.
    \]

  It remains to notice that the restriction of $\rey_W$ to
\(\R[\operatorname{M}_{m}]_{rm}\)
  is the sought-for $\Gamma$-equivariant projector  
  \[
    \rey_{(r)}
    \colon
\R[\operatorname{M}_{m}]_{rm}
    \to
    \R.D^r.
    \qedhere
  \] 
\end{proof}
\begin{rema}
  We haven't checked that the ``partial'' Reynolds operators \(\rey_{(r)}\)  constructed above,  are compatible with each other. By this, we mean that, for $r \leq s$, the restriction of \(\rey_{(s)}\)  to  \(D^{-r} \R[\operatorname{M}_{m}]\) (via the inclusion $D^{-r} \R[\operatorname{M}_{m}] \subset D^{-s}\R[\operatorname{M}_{m}]$) should be equal to  \(\rey_{(r)}\). This fact holds a posteriori, because complete irreducibility implies unicity of Reynolds operators (Remark \ref{unirey}). It  can also be checked directly.

\end{rema}

To conclude, we explain how to combine the previous results, to prove Theorem \ref{theo:rep_supp}. Let $V$ be a rational representation of $\Gamma$. Let $W \subset V$ be a subrepresentation.  By Example \ref{examFr} (applied to $\End(V)$), there exists $r\geq 1$, such that $\End(V) \in \mathcal F_r$. By Theorems \ref{theo:rey_Dr} and \ref{theo:rey_Fr}, we get the existence of Reynolds operators on  $\mathcal F_r$. It remains to cast Theorem \ref{theoReyCR}, to get a  supplementary subrepresentation of $W \subset V$.

\begin{rema}
  As the reader may check, the purely algebraic approach taken in Appendix A, just appeals to \textit{one} specific property of the field $\R$. Namely:  every non-zero integer is invertible in $\R$. As a consequence, all results thereof remain true (and are proved in the exact same way) for representations of $\GL(m)$ with coefficients in other \textit{fields of characteristic zero}, such as $\mathbb Q$ or  $\mathbb C$.

\end{rema}

\section{An orbit-limit lemma.}
This last appendix is devoted to prove the following important result of representation theory.
\begin{theo}
  Any finite-dimensional rational \(\GL(m)\)-representation is generated by its \(U_m\)-invariants.
\end{theo}

It will be a consequence of complete reducibility and of the following result.
\begin{lemm}
  \label{lemm:WU}
  Any non-zero finite-dimensional rational \(\GL(m)\)-representation admits a non-zero \(U_m\)-invariant.
\end{lemm}
\begin{proof}
  We use an orbit-limit method to give a constructive proof of a \(U_m\)-invariant.
  We have split the proof in three parts for more clarity.

  \noindent\textbf{a-} \underline{Filtration of \(U_m\).}
  We write elements of \(U_m\) as \(u=I+N\), with \(N\) strictly upper-triangular.
  For each row \(i\), define
  \[
    Q_i 
    \bydef
    \Set{I+N \colon N \text{ has support only in row } i} .
  \]
  Since the matrices \(E_{ij}\) with fixed \(i\) commute, the group law reduces to addition of coordinates; hence there is a well-known isomorphism \((Q_i,*) \simeq (\R^{i-1},+)\).
  Next, define subgroups inductively by \(S_0=\{I\}\) and \(S_{i+1}=Q_{i+1}S_{i}\).
  These groups are well-defined because for any \(g\in Q_{i+1}\) and \(u\in S_i\), one has \(gug^{-1}\in S_i\).
  One has
  \[
    S_i 
    \bydef
    \Set{I+N \colon N \text{ has support only in row } \leq i } .
  \]
  This yields an increasing filtration
  \[
    S_0 \subseteq S_1 \subseteq \dotsb \subseteq S_{m-1}=U_m.
  \]

  \noindent\textbf{b-} \underline{An orbit-limit lemma.}
  We will prove the following intermediary result.
  Let \(W\) be a finite-dimensional polynomial representation of \(\R^k\).
  Consider the induced action on lines \(\ell\subseteq W\).
  Then the lines \((t_1,\dotsc,t_k)\cdot\ell\) converge to a line of invariant points as \(\abs{t_i}\to+\infty\).

  We start with \(k=1\).
  Let \(\rho\colon\R\to\GL(W)\) denote the polynomial action of \(\R\) on \(W\), of degree \(\delta\).
  Then for \(s\neq0\), \(q(s,v)=s^\delta\rho(1/s)(v)\) is polynomial in \(s\) and colinear to \(\rho(1/s)(v)\).
  This allows to compute the limit of \([\rho(t)(v)]\) when \(\abs{t}\to+\infty\) as \([q(0,v)]\).
  This gives an invariant line because
  \[
    [\rho(t')(\rho(t)(v))]
    =
    [\rho(t'+t)(v)],
  \]
  and \(\rho(t')\) is continuous.
  The restriction of the \(\R\)-action to this line is an algebraic morphism \(\R\to\GL(\R)\), that is an algebraic power map, hence it is trivial.

  We can now extend the result to any \(k\) inductively by noticing that 
  \(
  \R^{k+1}=\R^k\times\R
  \),
  and that the actions of the two factors commute.
  Hence \(\R\) acts on \(W^{\R^k}\) and \((W^{\R^k})^{\R}=W^{\R^{k+1}}\).

  \noindent\textbf{c-} \underline{Construction of an invariant.}
  We can now finish the proof.
  Since the \(\GL(m)\)-representation \(W\) is rational, the restriction of the action to \(U_m\) is polynomial.
  Moreover for any \(g\in Q_{i+1}\) and \(u\in U_i\), one has \(g^{-1}ug\in U_i\), hence the action of \(Q_{i+1}\) on \(W\) induces an action of \(Q_{i+1}\) on the subspace \(W^{U_i}\) and \((W^{U_i})^{Q_{i+1}}=W^{U_{i+1}}\).

  Starting from \(v\in W\setminus\Set{0}\), we apply the previous lemma inductively by setting
  \(\ell_0\bydef[v]\in P(W)\) and then 
  \[
    \ell_{i+1}
    \bydef
    (Q_{i+1}\cdot\ell_i)_\infty
    \in
    P\big(W^{U_{i+1}}\big).
  \]
  We end up with a line \(\ell_{m-1}\) of \(U\)-invariant points.
\end{proof}

We can now finish the proof of the theorem.
Any finite-dimensional rational \(\GL(m)\)-representation decomposes as a direct sum of irreducible rational \(\GL(m)\)-representations,
see Appendix~\ref{sec:reducibility}.
Therefore, it is sufficient to prove the results for irreducible representations.
But an irreducible representation \(W\) is generated by any of its non-zero elements. We take any non-zero element of \(W^{U_m}\) (which is not reduced to \(\{0\}\) by the first point).

\end{document}